\documentclass[letter, 10pt]{article}
\usepackage{amssymb}
\usepackage{amsthm}
\usepackage{amsmath}
\usepackage{mathrsfs}
\usepackage{verbatim}
\usepackage{enumerate}

\newcommand{\R}{\mathbb{R}}
\usepackage{cite}
\usepackage{hyperref}
\usepackage{changes}
\usepackage{color}

\numberwithin{equation}{section}

\newtheorem{theorem}{Theorem}[section]
\newtheorem{theorem*}{Theorem}
\newtheorem{lemma}[theorem]{Lemma}
\newtheorem{corollary}[theorem]{Corollary}
\newtheorem{remark}[theorem]{Remark}
\newtheorem{proposition}[theorem]{Proposition}
\newtheorem{definition}[theorem]{Definition}

\newtheorem{example}[theorem]{Example}
\newtheorem{condition}[theorem]{Condition}
\newcommand\red{\color{red}}
\newcommand\blue{\color{blue}}

\newcommand\argmin{{\rm argmin}}
\newcommand\vol{{\rm vol}}
\newcommand\trace{{\rm tr}}
\newcommand\diam{{\rm diam}}
\newcommand\spt{{\rm spt}}
\newcommand\dist{{\rm dist}}

\begin{document}
\title{Wasserstein Barycenters over Riemannian manifolds \footnote{Y.-H.K. is supported in part by 
Natural Sciences and Engineering
Research Council of Canada (NSERC) Discovery Grants 371642-09 and 2014-05448 as well as Alfred P. Sloan research fellowship 2012-2016.  B.P. is pleased to acknowledge the support of a University of Alberta start-up grant and National Sciences and Engineering Research Council of Canada Discovery Grant number 412779-2012. Part of this research was done while Y.-H.K. was visiting Korea Advanced Institute of Science and Technology (KAIST) and while both authors were visiting the Mathematical Sciences Research Institute (MSRI), Berkeley, CA, and  the Fields Institute, Toronto, ON.
}}

\author{Young-Heon Kim\footnote{Department of Mathematics, University of British Columbia, Vancouver BC Canada V6T 1Z2 \ \ yhkim@math.ubc.ca} and Brendan Pass\footnote{Department of Mathematical and Statistical Sciences, 632 CAB, University of Alberta, Edmonton, Alberta, Canada, T6G 2G1 \ \ pass@ualberta.ca}}
\maketitle
\begin{abstract}
 We study barycenters in the space of probability measures on a Riemannian manifold, equipped with the Wasserstein metric.
Under reasonable assumptions, we establish absolute continuity of the barycenter of general measures $\Omega \in P(P(M))$ on Wasserstein space,  extending on one hand, results in the Euclidean case (for barycenters between finitely many measures)  of Agueh and Carlier \cite{ac} to the Riemannian setting, and  on the other hand, results in the Riemannian case of Cordero-Erausquin, McCann, Schmuckenschl\"ager \cite{c-ems} for barycenters between two measures to the multi-marginal setting.  Our work also extends these results to the case where $\Omega$ is not finitely supported.  As applications, we prove versions of Jensen's  inequality on Wasserstein space and a generalized Brunn-Minkowski inequality for 
a random measurable set 
on a Riemannian manifold.

\end{abstract}

\tableofcontents
\section{Introduction}\label{sec: intro}

This paper is devoted to the study of barycenters in Wasserstein space over a Riemannian manfold $M$.  

Given a Borel probability measure $\Omega$ on a metric space $(X,d)$, a barycenter of $\Omega$ is defined as a minimizer of $y \mapsto \int_{X}d^2(x,y)d\Omega(x)$; this definition is chosen in part so that it coincides with the mean, or center of mass, $\int_{\mathbb{R}^n}xd\Omega(x)$ on the Euclidean space $X = \mathbb{R}^n$.  Barycenters have been studied extensively by geometers, and have interesting connections to the underlying geometry of $X$; for example, their uniqueness is intimately related to sectional curvature.

The case where the metric space $(X,d)=(P(M),W_2)$ is the space of Borel probability measures on a compact Riemannian manifold $M$, equipped with the $W_2$ distance, is of particular interest,  as it gives a natural but nonlinear way to interpolate between a distribution of measures.   The Wasserstein, or optimal transport, distance $W_2 (\mu, \nu)$  between $\mu, \nu \in P(M)$ is given by

\begin{equation}\label{wass distance}
W_2^2(\mu,\nu) = \inf_{\gamma}\int_{M \times M}d^2(x,y)d\gamma(x,y)
\end{equation}
where $d$ is the Riemannian distance and the infimum is taken over all probability measures $\gamma$ on $M \times M$ whose marginals are $\mu$ and $\nu$.  It is well known that $W_2$ defines a metric on $P(M)$  (see, e.g., \cite{ags, V}) and therefore it makes sense to talk about barycenters. 

\begin{definition}[\bf Wasserstein barycenter measure]\label{def: W bary}
Let $\Omega$ be a probability measure on $P(M)$. A Wasserstein barycenter 
of $\Omega$
 is a minimizer among probability measures $\nu \in P(M)$ of 
\begin{align*}
 \nu \mapsto \int_{P(M)} W_2^2 (\mu, \nu) d\Omega(\mu).
\end{align*}
\end{definition}

Existence and uniqueness (under mild conditions) of Wasserstein barycenters are not difficult to establish (see Section~\ref{sec: exist and unique}).  When the measure $\Omega =(1-t)\delta_{\mu_0}+t\delta_{\mu_1}$ is supported on two points in $P(M)$, the barycenter measure $\mu_t$ on $M$ is equivalent to McCann's celebrated displacement interpolant \cite{m}. A key property of dispacement interpolants is that $\mu_t$ is  absolutely continuous with respect to volume if either$\mu_0$ or $\mu_1$ is; this fact plays a foundational role in the analysis of convexity type properties of various functionals on the space of absolutely continuous probability measures  $P_{ac}(M)$.  The notion of displacement convexity, or convexity of functionals along this interpolation, has been very fruitful; its  wealth of applications includes insightful new proofs of geometric and functional inequalities on $\mathbb{R}^n$, and remarkable generalizations of these inequalities to the Riemannian setting; see, e.g., the books \cite{ags, V, V2}.  In addition, displacement convexity is a fundamental notion in  the synthetic treatment of Ricci curvature  developed by Sturm \cite{sturm06}\cite{sturm06a} and Lott-Villani \cite{lottvillani}.     An important example of a displacement convex functional is  Boltzmann's $H$-functional (or the Shannon entropy functional) $\rho \mapsto \int_M \rho(x) \log \rho(x) dvol(x)$ on manifolds with nonnegative Ricci curvature. 

In the multi-measure setting, with economic applications in mind, Carlier-Ekeland \cite{CE} introduced an interpolation between several probability measures, which includes as a special case  Wasserstein barycenters of finitely supported measures $\Omega =\sum_{i=1}^m\lambda_i\delta_{\mu_i}$; in fact,  their setting is more general, as the distance squared in \eqref{wass distance} is replaced with a more general cost function.
  Agueh-Carlier \cite{ac} provided a more extensive treatment of Wasserstein barycenters of finitely supported measures  $\Omega \in P(\R^n)$  when the underlying space $M$ is Euclidean, establishing that the barycenter is absolutely continuous with an $L^{\infty}$ density if one of the marginals is, as well as a variety of convexity type inequalities, which one can interpret as Jensen's type inequalities for discrete measures and displacement convex functionals on $P(\R^n)$.   Absolute continuity has also been established for more general interpolations over Euclidean space, as well as barycenters on Hadamard manifolds  (simply connected Riemannian manifolds with nonpositive curvature) \cite{P9}.
  Some of these results have been extended  by one of us \cite{P5} to the case where the support of $\Omega \in P(\R^n)$ is parameterised by a $1$-dimensional continuum. 
  Let us note that in addition to economics \cite{CE}, Wasserstein barycenters in the multi-measure setting 
   have appeared in the literature with  applications in 
   image processing \cite{bdpr} and statistics \cite{BK}.

 In this paper, we consider the Wasserstein barycenters of general measures $\Omega \in P(P(M))$.  In particular, we allow the support of $\Omega$ to have cardinality greater than $2$ (and possibly be infinite) and the underlying domain $M$ to be a  general compact Riemannian manifold,  without any curvature or topological restrictions. At present, little is known about Wasserstein barycenters in this generality.

Our first main contribution  is to establish {\bf  absolute continuity with respect to volume measure of the  Wasserstein barycenter}, under reasonable conditions on the marginals: see
{\bf Theorems~\ref{thm: ac-for finite} and \ref{thm: ac for general}}.
In previous work  \cite{ac}, \cite{P5}, and \cite{P9},  regularity results on barycenters are obtained by exploiting either the special geometry of Euclidean space or the uniform convexity of the distance squared function (in the non-positively curved setting). 
These tools are not available in the general Riemannian setting, and our argument is based instead on approximations.  
Indeed, we first consider the case of a finitely supported $\Omega=\sum_{i=1}^m\lambda_i\delta_{\mu_i}$ on $P(M)$, and adapt an argument of Figalli-Juillet \cite{FJ} (who studied the two measure case on  the Heisenberg group and Alexandrov spaces); this amounts to approximating all but one of the measures $\mu_i$ by finite sum of Dirac measures, obtaining uniform estimates for the approximating barycenters and passing to the limit.    Once  absolute continuity of the barycenter $\bar \mu =\bar f dvol$  is known,  estimates  on the Jacobian determinants of the optimal maps from $\bar \mu$ to each $\mu_i \in spt(\Omega)$ (expressed in Theorem \ref{thm: detcontrol}) imply that the $L^\infty$ norm of the barycenter density $\bar f$ is controlled by the densities of the measures $\mu_i \in spt(\Omega)$ (see Theorem \ref{thm: upper bound of density} for a precise statement).  With this uniform control in hand, we are able  to treat the general case by approximating a general $\Omega$ on $P(M)$ by finitely supported measures:  see Theorem~\ref{thm: ac for general}.

Our estimates in Theorem \ref{thm: detcontrol} are a result of what we call \textit{first and second order balance} conditions (Theorem \ref{thm: bccondition}), reflecting the fact that the barycenter is a stationary point of the functional $\int_{P(M)} W_2^2 (\mu, \nu) d\Omega(\mu)$, and are expressed in terms of  what we call \emph{generalized}, or \emph{barycentric, distortion coefficients}.  These coefficients, roughly speaking, capture the influence of the barycenter operation on the volume of small sets in $M$, in the same way that the volume distortion coefficients, introduced by Cordero-Erausquin, McCann and Schmuckenschlager\cite{c-ems} capture the effect on the volume of a small set in $M$ by interpolation along geodesics with a common endpoint (in fact, the volume distortion coefficients employed there are precisely our generalized distortion coefficients in the case where $\Omega$ is supported on two measures).

 Using the above results, we are then able to establish certain {\bf Wasserstein Jensen's type inequalities} (see, in particular, {\bf Theorems~\ref{thm: k-Jensen} and \ref{thm: distortedconvexity}}) for a wide variety of displacement convex functionals on $P(M)$.
    In fact, we establish two distinct results of this type; one involves $k$-displacement convex functionals, and can be interpreted as a generalization of \cite[Theorem 17.15]{V2}.  The other expresses a distorted sort of convexity and involves our \emph{generalized, or barycentric, distortion coefficients}; this is closer in spirit to the line of research pioneered by Cordero-Erausquin, McCann and Schmuckenschlager\cite{c-ems}, and can be interpreted as a  generalization of \cite[Theorem 17.37]{V2}.   We note that geometric versions of Jensen's  inequality (that is, versions formulated in terms of barycenters on metric spaces rather than linear averages) are known for measures on finite dimensional smooth manifolds  \cite[Proposition 2]{EmMo} and on more general spaces with appropriate sectional curvature bounds (see \cite{sturm} and \cite{kuwae}); these sectional curvature bounds are not satisfied by Wasserstein space $(P(M), W_2)$ \cite{ags}. Before the present paper, a version of Jensen's  inequality  on $P(M)$, due to Agueh-Carlier\cite{ac},  was known only when the underlying space $M \subseteq \mathbb{R}^n$ is Euclidean and the measure $\Omega$ on $P(M)$ is finitely supported.  Of course, the type of distorted convexity in  \cite[Theorem 17.37]{V2} is peculiar to Wasserstein space, and even the statement of the corresponding Jensen's inequality (our \ref{thm: distortedconvexity}) requires a generalization of the classical volume distortion coefficients in \cite{c-ems}, which is formulated here for the first time.

Finally, as an application of the machinery developed in this paper, we offer {\bf a random version of the Brunn-Minkowski inequality on a Riemannian manifold: see Theorem~\ref{thm: random BM}.}    
 The classical Brunn-Minkowski inequality involves the interpolation between two sets in Euclidean space; an extension to Riemannian manifolds, with Ricci curvature playing a key role, can be derived directly from the results in \cite{c-ems}.  Our result extends this to  interpolations between random sets on $M$.  For a finite number of sets, we should note that in Euclidean space this result is easily recoverable using the classical Brunn-Minkowksi and induction.  For an infinite collection of sets, the Euclidean version is a pre-existing but nontrivial result, known as Vitale's random Brunn-Minkowski inequality \cite{vitale};  our work provides a mass transport based proof of it. On the other hand, in the Riemannian case, our result seems to be completely novel, as soon as we interpolate between three or more sets.

{\bf Organization of the paper:}
In the following section, we introduce the notation and terminology we will use throughout the paper and recall a few fundamental results from the literature which we will need.  In Section~\ref{sec: exist and unique}, we prove a  general existence and uniqueness results for the Wasserstein barycenter. 
 In Section~\ref{S:balance}, we establish some properties of the Wasserstein barycenter, including two balance conditions on Kantorovich potentials which will be crucial in subsequent sections. 
Section~\ref{S:regularity-finitely many}
 is devoted to the  proof of the absolute continuity of the Wasserstein barycenter when the measure $\Omega$ has finite support; this result is then exploited to prove absolute continuity of the barycenter of  more general $\Omega$ in Section~\ref{S:regularity-infinitely many}. 
This, in turn,  is used in Section~\ref{S:Jensen}, where we prove our  Wasserstein Jensen's  inequalities.  Finally, these results are exploited to establish a random Brunn-Minkowski inequality on curved spaces in Section~\ref{S:Brunn-Minkowski} . 

\section{Notation, definitions,  and preliminary results}\label{sec: preliminary}
In this section, we introduce some notation and terminology which we will use in the rest of the paper, and develop some preliminary results.  
\subsection{Notation and assumptions}\label{sec: notation}
Throughout the rest of the paper, we use the following notation and assumptions:
\begin{itemize}
 \item $M$ is a connected, compact $n$-dimensional Riemannian manifold. (The compactness assumption is made  primarily to keep the presentation relatively simple.  Most of the results in this paper can be established on non compact manifolds under suitable additional hypotheses, such as decay conditions on the measures, etc.)
   \item $d(x,y)$ is the Riemannian distance between two points $x,y$ in $M$.
 
  \item $B_r(x)$ is the geodesic ball of radius $r$ in $M$, centred at $x$.
  \item We will sometimes use the notation
$$c(x,y)=\frac{1}{2} d^2(x,y),$$

where $c$ stands for cost function. 
A significant property of the function $c$ is the following relation:
  $$  -D_ x c(x, y) = \exp_{x}^{-1} (y),$$
where $D_x$ denotes the gradient of $c$ with respect to the $x$-variable; although the notation $D_xc$ is often used to denote the differential of $c$ (a covector) rather than its  gradient (a vector), we will often identify vectors and covectors using the Riemannian metric.

 \item $P(M)$ is the space of probability measures on $M$ equipped with the weak-* topology, or, equivalently, metrized with the Wasserstein distance \eqref{wass distance}.
  
 \item  $\Omega$ is a probability measure on $P(M)$. 
 
\end{itemize}

 \subsection{Borel measurability of the set $P_{ac}(M)$}

 Equipped with the distance $W_2$, the space $P(M)$ is a separable metric space.  We consider measurable sets with respect to the Borel $\sigma$-algebra.

In this subsection we show that the set $P_{ac}(M)$ of absolutely continuous probability measures on $M$, with respect to the $n$-dimensional Hausdorff measure, or equivalently the Riemannian volume, is Borel measurable.  We expect that this is already known to experts, but we include it for completeness.
 
In  Section~\ref{sec: exist and unique}, when we show uniqueness of the Wasserstein barycenter of a given probability measure $\Omega$ on $P(M)$, we will need to assume  $\Omega (P_{ac}(M))>0$.

\begin{proposition}[\bf Measurability of $P_{ac}(M)$]\label{prop: Borel P ac}
The set $P_{ac}(M) \subset P(M)$, of absolutely continuous probability measures is Borel measurable with respect to the metric topology given by the Wasserstein distance $W_2$, or equivalently with respect to the weak-* topology (the two topologies are equivalent \cite{V2}).
\end{proposition}
\begin{proof}
Note that  absolute continuity of a measure $\mu$  (with respect to $\vol$) is equivalent to the following property: 
 for every  $\epsilon >0$, there is  $\delta>0$ such that $\mu (A) \le \epsilon$ for all Borel sets $A$ with $\vol(A) \le \delta$. 
 This means 
\begin{align}\label{eq: P ac cap cup}
 P_{ac}(M) =  \cap_{k\in \mathbf{N}}\cup_{l \in \mathbf{N}}\mathcal{E}_{2^{-k},2^{-l}},
\end{align}
where the sets $\mathcal{E}_{\epsilon, \delta}$ of probability measures are defined as 
\begin{align*}
 \mathcal{E}_{\epsilon, \delta}& = \{ \mu \in P(M) \ |  \ \mu (A) \le \epsilon \\ & \qquad  \hbox{ for all Borel set  $A$ with  $\vol (A)  \le \delta$}\}.
\end{align*}

To show $P_{ac}(M)$ is a Borel set, we will express it  as a countable intersection of  countable unions of Borel sets (essentially replacing the set $\mathcal{E}_{2^{-k},2^{-l}}$ in  \eqref{eq: P ac cap cup} with a closed set).  
For this, 
first define 
\begin{align*}
 \mathcal{F} &= \{ F \subset M \ | \ F = \cup_{i=1}^m B_{r_i} (x_i) \\ & \qquad 
 \hbox{ for some finite sets $\{x_i\}_{i=1}^m \subset M$ and $\{r_i\}_{i=1}^m \in \R_+$}\}
\end{align*}
and consider the subset $\mathcal{B}_{\epsilon, \delta} \subset P(M)$ defined as 
\begin{align*}
 \mathcal{B}_{\epsilon, \delta} = \{ \mu \in P(M) \ | \  \mu (F) \le  \epsilon, \quad \forall F \in \mathcal{F} \hbox{ with } \vol(F) \le \delta \}.
\end{align*}
We claim that the set $\mathcal{B}_{\epsilon, \delta}$ is a closed subset of $P(M)$, with respect to the weak-* topology.  To see this, pick any sequence $\mu_i\in \mathcal{B}_{\epsilon,  \delta}$,  weakly-* convergent 
to $\mu_\infty$. Pick a  set $F \in \mathcal{F}$, with $\vol (F) \le \delta$ and $F = \cup_{i=1}^m B_{r_i} (x_i)$. For each $k \in \mathbf{N}$, let $F_k = \cup_{i=1}^m B_{(1-2^{-k})r_i} (x_i)$ and consider a continuous function $f_k: M \to \R$, with support 
  $\spt f_k  \subset F$  and $0\le f_k\le 1$ and $f_k= 1$ on $F_k$. Then, due to the weak-* convergence $\epsilon \ge \lim_{i \to \infty } \int_M f_k d\mu_i = \int_M f_k d\mu_\infty$. Moreover, $\mu_\infty (F) = \lim_{k \to \infty} \int_M f_k d\mu_\infty$ since $F = \cup_{k} F_k$. This shows that $\mu_\infty (F) \le \epsilon$, as desired.

  Now, clearly $\mathcal{E}_{\epsilon, \delta} \subset \mathcal{B}_{\epsilon, \delta} $. 
  We show that $\mathcal{B}_{\epsilon, 2\delta} \subset \mathcal{E}_{\epsilon, \delta}$: Let $\mu \in \mathcal{B}_{\epsilon, 2\delta}$. Let $A$ be an arbitrary Borel set with $\vol(A)\le \delta$. Pick an arbitrary small number $t >0$.  One can find an open set $U_{A, \delta, t}$, 
consisting of finite metric balls $U_{A, \delta, t}= \cup_{i=1}^m B_{r_i} (x_i)$, with $\vol(U_{A, \delta, t}) \le 2\delta$ and $\mu (U_{A, \delta, t}) \ge (1-t) \mu (A)$.    Then, from the definition of $\mathcal{B}_{\epsilon, 2\delta}$,   $\mu (U_{A, \delta, t})\le  \epsilon$, therefore, $\mu(A) \le \frac{1}{1-t} \epsilon$. Since $t>0$ was arbitrary, this means $\mu(A) \le \epsilon$. 
This shows $\mathcal{B}_{\epsilon, 2\delta} \subset \mathcal{E}_{\epsilon, \delta}$.
   
   The above paragraph implies
\begin{align*}
 P_{ac}(M) &=  \cap_{k\in \mathbf{N}}\cup_{l \in \mathbf{N}}\mathcal{E}_{2^{-k},2^{-l}}\\
 &= \cap_{k\in \mathbf{N}}\cup_{l \in \mathbf{N}}\mathcal{B}_{2^{-k},2^{-l+1}}
\end{align*}
   which completes the proof, since the latter expression is a countable union of intersections of closed (thus Borel) sets. 
\end{proof}

\begin{remark}
 Inspection of the  proof above shows that $(M, \vol)$ can be replaced with a compact separable metric space $(X, \nu)$  equipped with a reference Borel measure $\nu$. 
\end{remark}

\subsection{Optimal transport on Riemannian manifolds}

Next, we briefly recall some key results in optimal transport on Riemannian manifolds which we will use throughout the paper.  We begin with a fundamental result of McCann \cite{m3}, orignally established by Brenier \cite{bren} when $M=\mathbb{R}^n$ is Euclidean:
\begin{theorem}[\bf Optimal transport on $M$; see Brenier \cite{bren}, McCann  \cite{m3}]\label{thm: McCann}
Assume $\mu$ is absolutely continuous with respect to volume measure.  Then the infimum in \eqref{wass distance} is attained by a unique measure $\gamma$.  Furthermore, $\gamma = (I \times T)_\#\mu$ is concentrated on the graph of a  measurable mapping $T$ over the first marginal, and $T$ takes the form

$$
T(x) = \exp_x(Du(x))
$$
where $u:M \rightarrow \mathbb{R}$ is a $c$-convex funtion; that is
\begin{align}\label{eq: c-convex}
 u(x) = \sup_{y \in M}- c(x,y)-u^c(y)
\end{align}
for some function $u^c:M \rightarrow \mathbb{R}$, where $c(x,y) = \frac{1}{2} d^2 (x, y)$. 
\end{theorem}

It is well known that the $c$-convex function $u$  is  semi-convex and therefore twice differentiable almost everywhere.  At each point where this differentiability holds, the mapping $T$ is differentiable.  We now recall a few classical identities, easily derived from the Brenier-McCann theorem (Theorem~\ref{thm: McCann}), or, for more general cost functions, from references such as \cite{GM} \cite{Caf} and \cite{mtw}.

Wherever $u$ is differentable, we have the first order condition:
\begin{align*}
- D_x c (x, T(x)) = D_x u (x).
\end{align*}
Differentiating this identity, we obtain
\begin{equation}\label{eq: c-equation}
D^2u(x) + D^2_{xx}c(x,T(x)) = -D^2_{xz}\Big|_{z=T(x)}c(x,z)\cdot DT(x).
\end{equation}
It will also be  important later to recall the second-order inequality  due to \eqref{eq: c-convex}:
\begin{align}\label{eq: D2 u monotone}
 D^2u(x) + D^2_{xx}c(x,T(x)) \ge 0.
\end{align}
Taking determinants of \eqref{eq: c-equation} yields
\begin{align}\label{eq: det}
&| \det[ DT(x)]| 
 \\\nonumber &=| \det [D^2_{xz} c (x, T(x)]^{-1} \det[D^2 u (x) + D^2_{xx} c (x, T(x)]|.
 \end{align}
If both $\mu$ and $\nu$ are absolutely continuous with respect to volume, with densities $f$ and $g$, respectively, we also have the change of variables formula almost everywhere:
\begin{align}\label{eq: Jacobian}
&g(T(x)) |\det [DT(x)]| = f(x),
\end{align}
which, together with \eqref{eq: det}, implies
\begin{align}\label{eq: c-MA}
& g  (T(x))|\det [D^2_{xy} c (x, T(x)]^{-1}  \det(D^2u(x) + D^2_{xx}c(x,T(x)) )|\\\nonumber
 &= f(x).
\end{align}
In the present paper we will also make use of a multi-marginal version of the Brenier-McCann theorem (Theorem~\ref{thm: McCann}), which generalizes from Euclidean space  a result of Gangbo and Swiech \cite{GS} and is also related to the works of Carlier-Ekeland \cite{CE} and Agueh-Carlier \cite{ac}.  Given probability measures $\mu_1,...,\mu_m$ on $M$, the multi-marginal optimal transport problem is to minimize

\begin{equation}\label{K}
\int_{ \Pi_{i=1}^m M}c(x_1,...,x_m)d\gamma
\end{equation}
over all  probability measures $\gamma$ on  the $m$-tuple product $\Pi_{i=1}^m M^m$ whose marginals are  the $\mu_i$'s. There has recently been substantial interest and progress in understanding this problem in a variety of different settings; see \cite{Pass14} and the references therein.
In this paper,  we will take the cost function $c: \Pi_{i=1}^m M \rightarrow \mathbb{R}$ to be 
\begin{equation}\label{bccost}
c(x_1,...,x_m) = \min_{z \in M}\sum_{i=1}^m\lambda_id^2(x_i,z)
\end{equation}
 where $\lambda_1, ..., \lambda_m > 0$, with $\sum_{i=1}^m \lambda_i =1$, represent weights on the components $d^2(x_i,z)$ making up the cost function; we will sometimes denote $\vec \lambda = (\lambda_1, ..., \lambda_m) $.
In Euclidean space, this coincides with the cost studied by Gangbo and Swiech \cite{GS}, who proved   assertion 1 in Theorem~\ref{thm: Kim Pass} below in that setting, extending earlier partial results of Olkin and Rachev \cite{OR}, Knott and Smith \cite{KS} and Ruschendorf and Uckelmann \cite{RU}.  

We  have the following theorem:
\begin{theorem}[{\bf Multi-marginal optimal transport on $M$; see \cite[Sections 4 and 5]{KP}}]\label{thm: Kim Pass}
Assume $\mu_1$ is absolutely continuous with respect to $\vol$.

\begin{enumerate}
 \item  The solution $\gamma$ to \eqref{K} with cost function \eqref{bccost} is concentrated on the graph of a mapping $(F_2,F_3,...,F_m)$ over the first variable and is unique.  

\item   There exists a unique minimizer  $\bar x_{\lambda}(x_1,x_2,...,x_m)$  of $x \mapsto \sum_{i=1}^m \lambda_i d^2(x_i,y)$  for $\gamma$ almost all $(x_1,x_2,...,x_m)$, and moreover this gives a $\gamma$-a.e one-to-one map $\bar x_{\lambda}: \spt \gamma \to M$.

\item Moreover, applying 1 and 2 to a result of Carlier-Ekeland \cite[Proof of Proposition 3]{CE}, we get
\begin{equation*}
\nu:=\overline{x}_{\lambda\#} \gamma
\end{equation*}
is the unique Wasserstein barycenter measure of the measures $\mu_1, ..., \mu_m$ with weights $\lambda_i$.

\end{enumerate}

\end{theorem}
In particular, the assertions 2 and 3 will be important for us.   
\subsection{Geometric barycenters on Riemannian manifolds: volume distortion}
In the remainder of the present section, we discuss  geometric barycenters on a Riemannian manifold and introduce the volume distortion constants associated to them.
Given a probability measure $\lambda$ on $M$, we denote its set of barycenters by
$$
BC(\lambda): = \argmin \big( y \mapsto \int_{M} c(y,x)d\lambda(x) \big)
$$
 We introduce the notation:
\begin{align*}
 bc_\lambda (x_1, ..., x_m) = BC \left(\sum_{i=1}^m \lambda_i \delta_{x_i}\right),
\end{align*}
for the barycenter of the discrete measure with weights $\lambda_1, ..., \lambda_m >0 $.

We will also require the following notion:
\begin{definition}[\bf Volume distortion]
Let $\lambda$ be a Borel probability measure on $M$ with a unique barycenter $\bar x$ (that is, such that $BC(\lambda) $ is a singleton).  We define the genralized, or barycentric, volume distortion coefficients  at $y \notin cut(\bar x)$
\begin{align}\label{eq: alpha}
\alpha_{\lambda}(y) :=\frac{\det[-D^2_{yz}\big|_{z=\bar x} c(y, z)]}{\det [ \int_M D^2_{zz} \big|_{z=\bar x}c(x, z) d\lambda(x)]}
\end{align}
where $D^2_{zz} c (x, z)$ denotes the Hessian of the function $z\mapsto c(x, z)$,
and the determinants are computed in  exponential local coordinates at $\bar x$ and $y$.  
\end{definition}

\begin{remark}[\bf Justification of the name volume distortion for $\alpha_{\lambda}$]
Volume distortion coefficents were introduced in \cite{c-ems}; they capture the way the volume of a small ball is distorted as it is slid along geodesics ending at a common fixed point.

Our generalized coefficients, roughly speaking, capture the way that the volume of a small ball centred at $x \in spt(\lambda)$ is distorted by interpolating between points in this ball and the other points in the support of a probability measure $\lambda$ on $M$; the classical coefficents correspond to the case when $\lambda$  is concentrated at two points.  We make this analogy precise below, in the case that $\lambda$ is finitely supported.

Suppose  
\begin{enumerate}
\item 
 $\lambda = \sum_{i=1}^m \lambda_i \delta_{x_i}$ has finite support and assume that, for $y$ near $x_j$, $$ BC({ \sum_{i\neq j}^m \lambda_i \delta_{x_i}}+\lambda_j \delta_y )$$  is a singleton;
 \item the function $ y \mapsto BC({ \sum_{i\neq j}^m \lambda_i \delta_{x_i}}+\lambda_j \delta_y )$ is differentiable at $x_j$.  
 \end{enumerate}
 
 We claim that, for a fixed index $j$,
$$
\alpha_{\lambda}  (x_j) =\lim_{r \rightarrow 0}   \frac{\vol (BC(\lambda, B_{ r}(x_j) ))}{\vol (B_{\lambda_j r}(x_j))}
$$
where $BC(\lambda, B_{ r}(x_j) ) = \cup_{y \in B_{ r}(x_j) } BC({ \sum_{i\neq j}^m \lambda_i \delta_{x_i}}+\lambda_j \delta_y )$.

In particular, when $\lambda = t\delta_x + (1-t) \delta_y$, we have $\alpha_{\lambda}(x) =v_{1-t}(y,x)$, where $v_{1-t}(y,x)$ is the volume distortion coefficient of \cite{c-ems}.
\end{remark}
\begin{proof}
 From  assumption 1,  we can define $\bar x(x_1,..., x_m) =BC({ \sum_{i=1j}^m \lambda_i \delta_{x_i}}) $. Now, the function   $$z \mapsto \sum_{i=1}^m\lambda_i  c(x_i,z ) $$ is  differentiable near $z=\bar x(x_1,...x_m)$ (for a proof, see e.g. \cite[Lemma 3.1]{KP}); moreover,  by minimality, we have
$$\sum_{i=1}^m\lambda_i D_{z}\Big|_{z=\bar x} c(x_i,z)=0.$$
 From  assumption 2, we can differentiate the last equation with respect to $x_j$, which  yields
$$
\sum_{i=1}^m\lambda_i D^2_{zz}c(x_i,\bar x) \cdot D_{x_j} \bar x+ \lambda_j D^2_{z x_j}c(x_j, \bar x) =0.
$$
After taking determinants and rearranging, we have,
\begin{align}\label{eq: det ratio}
\det (D_{x_j} \bar x) & = \frac{\lambda_j^{n} \det[ -D^2_{z x_j}c(x_j, \bar x) ]}{\det[\sum_{i=1}^m\lambda_i D^2_{zz}c(x_i,\bar x)]}
\end{align}
Notice that the absolute value of the left-hand side of \eqref{eq: det ratio} is the volume distortion $$\lim_{r \rightarrow 0} \frac{\vol(BC(\lambda, B_{ r}(x_j) ))}{\vol (B_{ r}(x_j))};$$ since all the terms on the righthand side are nonnegative (see Lemma \ref{expcon} below),  dividing \eqref{eq: det ratio} by $\lambda_j^n$ yields the desired result.
\end{proof}
Before concluding this section, we prove a result relating the $\alpha_{\lambda}$ to the Ricci curvature of $M$.
 Let us fix the notation:
\begin{align}\label{eq: L and S K}
S_K (d) & =\left \{
  \begin{array}{l l}
  \frac{  \sin(\sqrt{K}d)}{\sqrt K d}& \quad \text{if $K>0$}\\
   1 & \quad \text{if $K=0$}\\
  \frac{  \sinh(\sqrt{-K}d)}{\sqrt{- K} d}& \quad \text{if $K<0$}
  \end{array}
\right .
\end{align}
We will need the following lemma, whose proof is based on an argument in \cite{c-ems}.
 \begin{lemma}\label{hessbounds}
Suppose $-K \leq 0$ is a lower bound for the Ricci curvature on $M$.  Then $$\trace [(D_{xx}^2c(x,y))] \leq n
\frac{\sqrt{K}d(x,y)}{\tanh(\sqrt{K}d(x,y))}.
 $$  
\end{lemma}

The proof is exactly as in \cite[Lemma 3.12]{c-ems}, but we take the trace over a orthonormal basis to get  from sectional to Ricci curvature.

\begin{lemma}\label{expcon}
Suppose $K$ is a lower bound for the Ricci curvature on $M$. Then 
$$
\det(-D_{xy}^2c(x,y)) \geq  [S_{K}(d(x,y))]^{-n+1}
$$
where $S_K$ is given in \eqref{eq: L and S K}.

\end{lemma}

\begin{proof}
Note that $-D_{xy}^2c(x,y)$ is the inverse of $d\exp_x(\cdot)$ evaluated at $-D_xc(x,y)$.  Therefore, by the  Bishop-Gromov volume comparison theorem (see, e.g., \cite{BiCr}, section 11.10, Theorem 15), 
$$t \mapsto \det(-D_{xy}^2c(x,y_t)) \cdot [S_{K}(d(x,y_t))]^{n-1}$$
is nondecreasing along a geodesic $y_t$ starting at $y_0=x$. As this function is $1$ when $t=0$, the result follows.
\end{proof}

\begin{proposition}[\bf Distortion under $Ric\ge 0$]\label{prop: alpha bound}
Suppose the Ricci curvature of $M$ is everywhere nonnegative, i.e., $Ric \geq 0$.  Then, for any $x \in M$ and $\lambda \in P(M)$, we have
$$
\alpha_{\lambda}(x)  \geq 1.
$$
\end{proposition}
\begin{proof}
{Minimality of $z \mapsto \int_M c(x,z) d\lambda(x)$ at the barycenter $\bar x$, combined with semi-concavity of  $z \mapsto c(x,z)$ and Fatou's lemma yields $$\int_M D^2_{zz} \big|_{z=\bar x}c(x, z) d\lambda(x) \geq 0$$ as a matrix (notice that until this moment we do not need any assumption on the curvature).  Now, as $\lambda$ is a probability measure,   Lemma \ref{hessbounds} with $K=0$ implies 
$$
\trace [\int_M D^2_{zz} \big|_{z=\bar x}c(x, z) d\lambda(x)] \le n;
$$
applying the geometric-arithmetic mean inequality to the nonnegative matrix $\int_M D^2_{zz} \big|_{z=\bar x}c(x, z) d\lambda(x)$ yields 
$$
\det [\int_M D^2_{zz} \big|_{z=\bar x}c(x, z) d\lambda(x)] \le 1.
$$
Combining this  with the inequality
$\det [-D_{yz}^2 \big|_{z=\bar x} c(y, z)] \ge 1$
(from Lemma \ref{expcon} with $K=0$), yields the desired result.  }
\end{proof}

{More generally, if $Ric \geq -K$ (for $K\ge0$), we have 
\begin{align}\label{eq: Ric -K case}
\alpha_{\lambda}(x) \ge C(\diam(M), K, n)
\end{align} where 
$\diam(M)$
is the diameter of the manifold and 
\begin{align*}
& C(\diam(M), K, n) \\&:=
\begin{cases}
    1  & \text{for $K=0$}, \\
     \left(S_{-K}(\diam (M))^{-n+1}\cdot   \frac{\sqrt{K}\diam(M)}{\tanh(\sqrt{K}\diam(M))} 
      \right)^{-n},  & \text{for $K>0$}.
\end{cases}
\end{align*}
}

\section{The Wasserstein barycenter: existence and uniqueness}\label{sec: exist and unique}

Let us  recall the notion of a Wasserstein barycenter of a probability measure $\Omega$ on $P(M)$ (Definition \ref{def: W bary} in the introduction).   Wasserstein barycenters were considered previously by Agueh-Carlier, who established existence and uniqueness results for finitely supported measures $\Omega \in P(P(\mathbb{R}^n))$ when the underlying space is Euclidean \cite{ac}.  Other variants of these results can be found in \cite{CE} \cite{P5} and \cite{P9}.

We present below a general existence and uniqueness result, which encompasses the earlier results found in \cite{ac} \cite{CE} \cite{P5} and \cite{P9}.  The proof is essentially the same as the argument found in \cite{P5}, but is included in the interest of completeness.
\begin{theorem}[\bf Existence and uniqueness of the Wasserstein barycenter]\label{thm: existence and uniqueness}
 Recall the assumptions and notation in Section~\ref{sec: notation}
 If $\Omega(P_{ac}(M)) >0$, then there exists a unique Wassertein barycenter of $\Omega$. 
\end{theorem}

\begin{proof}
 Due to compactness of $M$, the set $P(M)$ of probability measures on $M$ is weak-* compact, or, equivalently, the Wasserstein space $(P(M), W_2)$ is compact. 
Now, for any $\mu$, the mapping $\nu \mapsto W_2^2(\mu,\nu)$ is uniformly Lipschitz on Wasserstein space, and therefore so too is $\nu \mapsto \int_{P(M)}W_2^2(\mu,\nu)d\Omega(\mu)$.  Therefore, existence of a minimizer follows immediately. 

The uniqueness will follow from the fact that, with respect to linear interpolation of measures, the functional $\nu \mapsto \int_{P(M)}W_2^2(\mu,\nu)d\Omega(\mu)$ is convex and the convexity is strict if $\Omega(P_{ac}(M)) >0$.  
We prove this below.

We begin by studying  $\nu \mapsto W_2^2(\mu,\nu)$.  Let $\nu_0, \nu_1  \in P(M)$.  Let $\gamma_{i}$ be optimal couplings between $\nu_i$ and $\mu$, for $i =0,1$, respectively.  We set $\nu_s = s\nu_1 +(1-s)\nu_0$ and  $\gamma_{s} = s\gamma_{1} +(1-s)\gamma_{0}$.  Noting that $\gamma_{s}$ has $\nu_s$ and $\mu$ as its margnals, we have
\begin{eqnarray}
W_2^2(\mu,\nu_s) &\leq& \int_{M \times M} d(x ,y)^2d\gamma_{s} \nonumber\\
&=& s\int_{M \times M} d(x ,y)^2d\gamma_{1} +(1-s)\int_{M \times M}d(x ,y)^2d\gamma_{0}\nonumber\\
&=& sW_2^2(\mu,\nu_1)+(1-s)W_2^2(\mu,\nu_0)\label{convexity}
\end{eqnarray}
This yields convexity of the function $\nu \mapsto W_2^2(\nu,\mu)$.  

Next, we will show this convexity is strict if $\mu \in P_{ac}(M)$.    By the Brenier-McCann theorem (Theorem \ref{thm: McCann}), there exists a unique optimal map $F_s: spt(\mu) \rightarrow spt(\nu_s)$ for each $s$, such that the unique optimal measure $\overline{\gamma_{s}}\in \Gamma(\mu,\nu_s)$ is concentrated on the graph $\{(x,F_s(x)\}$. 

We need to show that, assuming $\nu_0 \neq \nu_1$ and  $0< s < 1$, the inequality \eqref{convexity} is strict.  Note first that the inequality is strict unless $\gamma_s$ is an \textit{optimal} coupling between $\mu$ and $\nu_s$; by the uniqueness result, this means we must have $\gamma_s= \overline{\gamma_{s}} =(Id,F_s)_\# \mu$.  That is, $\gamma_s$ is concentrated on the graph of $F_s$.

On the other hand, $\gamma_s$ is concentrated on the union of two graphs, $F_0(x)$ and $F_1(x)$:
$$
\gamma_s = s(Id,F_1)_\# \mu +(1-s)(Id,F_0)_\# \mu.
$$
This is possible only if $F_0=F_1=F_s$ $\mu$ almost everywhere, which, in turn, implies $ \nu_0=(F_0)_\#\mu =(F_1)_\#\mu=\nu_1$.  This yields strict convexity of $\nu \mapsto W_2^2(\nu,\mu)$ whenever $\mu$ is absolutely continuous with respect to volume.

 Finally, integrating $\nu \mapsto W_2^2(\nu,\mu)$ with respect to $\Omega$   yields  convexity of    the functional $\nu \mapsto \int_{P(M)}W_2^2(\mu,\nu)d\Omega(\mu)$, and the convexity is strict  under the assumption $\Omega(P_{ac}(M))>0$. This implies uniqueness of its minimizer, the Wasserstein barycenter of $\Omega$. 

\end{proof}
\begin{remark}
 By inspecting the above proof, it is clear that Theorem~\ref{thm: existence and uniqueness} holds for more general spaces than  Riemannian manifolds. In fact, it holds for any (compact) metric space on which  the optimal maps, $T_\# \mu  = \nu$, exist uniquely for any arbitrary absolutely continuous source measure $\mu$.  This includes for example, Alexandrov spaces \cite{bertrand}.%
\end{remark}
\begin{example}
 As an illuminating example, consider the round sphere. 
 If the $\Omega =\frac{1}{2}[ \delta_{p_n} + \delta_{p_s}]$ be the sum of two Dirac measures supported on the north and south pole, then its Wasserstein barycenter is not unique: any probability measure supported on the equator is a  Wasserstein  barycenter. 
 However, if we smear out one of the Dirac measures making it absolutely continuous, then the resulting Wasserstein barycenter will be a unique (in fact absolutely continuous) measure supported near the equator. 
\end{example}
\section{Properties of the Wasserstein barycenter: first and second order balance}\label{S:balance}

{
We develop here several properties of the Wasserstein  barycenter which we will use later on.  For some of these, we will need to assume that the Wasserstein barycenter is absolutely continuous with respect to volume.  Conditions on $\Omega$ ensuring this  absolutely continuity will be presented later on. The main results of this section are Theorems~\ref{thm: bccondition} and \ref{thm: detcontrol}, which are crucial for later sections.

\subsection{Differentiability of family of dual potentials}\label{SS: family of dual potentials}
 The key results of this subsection are \eqref{eq: derivative integrals} and \eqref{eq: 2nd deriv integral} for derivatives of the integral of a measurable family of dual potentials for optimal transport problems. We first establish the almost everywhere second differentiability of a certain measurable family of dual potentials:
\begin{lemma}[\bf a.e. $x$ and $\Omega$-a.e. $\mu$]\label{lem: a.e. diff}

Let $\bar \mu \in P(M)$ and for each $\mu \in P(M)$, let $u_\mu$ be the dual potential (determined modulo an  additive constant)  for the optimal transport problem \eqref{wass distance} between $\bar \mu$ and $\mu$.  Let $\Omega$ be a Borel probability on $P(M)$.  
For volume almost all $x$, $x \mapsto u_\mu(x)$ is twice differentiable for $\Omega$-almost all $\mu \in P(M)$.
\end{lemma}
\begin{proof}
The proof is a simple application of Fubini's theorem. Let $A \subset P(M) \times M$ be the set of points where the twice differentiability fails.  We are to show that its projection onto $P(M)$, namely, $A_x =\{\mu:(\mu,x) \in A \}$ has $\Omega$-measure zero, for almost all $x$ (notice that the set $A$ is measurable).  Assume by contradiction that  $A_x$ has positive $\Omega$-measure for some non-measure zero set of  $x$. Then there exists $\epsilon > 0$ and a set $B \subseteq M$ with $|B| >0$ such that $\Omega(A_x) \geq \epsilon$ for all $x \in B$.  Therefore, using Fubini's theorem, we have
\begin{align*}
(\Omega \times vol)(A) &= \int_M\int_ {P(M)} \chi_Ad\Omega dvol
 \geq \int_B\int_ {P(M)} \chi_A d\Omega dvol\\
 &= \int_B\int_ {P(M)} \chi_{A_x}  d\Omega dvol
=\int_B \Omega(A_x)dvol(x)\\
&\geq vol(B) \epsilon >0
\end{align*}
On the other hand, for each $\mu$, the dual potential $u_\mu$ is a semi-convex function \cite{c-ems} (recall $M$ is compact), so due to Alexandrov's second differentiability theorem,  $u_\mu$ is twice differentiable at Lebesgue (vol) a.e. points; i.e.  $A_\mu=\{x:(\mu,x) \in A \}$ has zero volume.  So we have
\begin{eqnarray*}
(\Omega \times vol)(A) = \int_{P(M)}\int_ M \chi_Advol d\Omega =\int_{P(M)}\int_ M \chi_{A_\mu} dvol d\Omega
 = \int_{P(M)}0 d\Omega=0.
\end{eqnarray*}
The contradiction implies the desired result
\end{proof}

From Lemma~\ref{lem: a.e. diff}, we see that for  Lebesgue almost every 
 $x$,  the maps $ \mu \mapsto \nabla_x u_\mu (x) $ and $ \mu \mapsto \nabla^2_x u_\mu(x)$  are well defined $\Omega(\mu)$ a.e.
{Now, an essential ingredient in our work is the function $x \mapsto \int_{P(M)}  u_\mu (x) d\Omega (\mu)$, where $u_\mu$ is  the dual potential function 
given in Lemma~\ref{lem: a.e. diff}.  Note that this function is Lipschitz and semi-convex since each $u_\mu$ is uniformly Lipschitz and semi-convex  (recalling that $M$ is compact).
By Rademacher's theorem and Alexandrov's second differentiability theorem, this function is twice differentiable for Lebesgue almost every $x$.  Moreover, applying Lemma~\ref{lem: a.e. diff},  for almost every $x$, 
we immediately have the following:
\begin{proposition}[{\bf Derivatives inside the integral $\int_{P(M)} d\Omega$}]\label{prop: derivatives inside integral}
\begin{align}\label{eq: derivative integrals}
\nabla_x \int_{P(M)}  u_\mu (x) d\Omega (\mu)&=  \int_{P(M)} \nabla_x u_\mu (x) d\Omega (\mu), \\\label{eq: 2nd deriv integral}
\nabla^2_x \int_{P(M)}  u_\mu (x) d\Omega (\mu)& \ge  \int_{P(M)} \nabla^2_x u_\mu (x) d\Omega (\mu).
\end{align}
\end{proposition}
\begin{proof}
 This can be seen by applying the dominated convergence theorem for \eqref{eq: derivative integrals} due to uniform Lipschitzness of $u_\mu$  and Fatou's lemma for \eqref{eq: 2nd deriv integral} due to the semi-convexiy of $u_\mu$. 
\end{proof}
}
\subsection{First and second order balance at the Wasserstein barycenter}\label{SS: balance}
 We now consider the Wasserstein barycenter measure $\bar \mu$ of $\Omega \in P(P(M))$, and the dual potentials $u_\mu$ for optimal transport problems \eqref{wass distance} from $\bar \mu$ to $\mu$.
Using the equations \eqref{eq: derivative integrals} and \eqref{eq: 2nd deriv integral}, we will establish the main results of this section, namely, the first and second order balance between the $u_\mu$'s with respect to $\Omega$: Theorem~\ref{thm: bccondition}.
We begin with a lemma relating barycenters on the manifold $M$ to  Wasserstein barycenters on $P(M)$:
\begin{lemma}[\bf Riemannian barycenter from Wasserstein barycenter]\label{riemannianbc}
Let $\bar \mu$ be a Wasserstein barycenter of the measure $\Omega$ on $P(M)$ and assume $\bar \mu$ is absolutely continuous with respect to volume; let $T_\mu$ be an optimal map from $\bar \mu$ to $\mu$.  Let $\lambda_z =(\mu \mapsto T_\mu(z))_\#\Omega$.  Then, for $\bar \mu$ almost every $z$, $z$ is a barycenter of $\lambda_z$.  

If, in addition, $\Omega(P_{ac}(M)) >0$, then for $\bar \mu$ almost every $z$, $z$ is the \emph{unique} barycenter of $\lambda_z$.
\end{lemma}
\begin{proof}
We first show that $\bar\mu$-a.e. $z$  is a barycenter $\lambda_z$. The proof is by contradiction; suppose not.  Then there exists a set $A \subset M$ with $\bar \mu (A) >0$ and for all $z \in A$, z is \emph{not} a barycenter of $\lambda_z$. 
 We define $g:\spt(\bar \mu)\mapsto M$ by letting $g(z) \in BC(\lambda_z)$ be a measurable selection of the barycenters. 
 Then, for all $z \in\spt(\bar \mu) $, we have
\begin{equation}\label{redvar}
\int_{P(M)}d^2(z,T_\mu(z))d\Omega(\mu) \geq \int_{P(M)}d^2(g(z),T_\mu(z))d\Omega(\mu) 
\end{equation}
and the inequality is strict on the set $A$ of positive $\bar \mu$ measure.  We define the probability measure.  
$$
\nu :=  g_\#\bar \mu.
$$
For each $\mu$, the measure $\gamma_\mu:=(g,T_\mu)_\#\bar \mu$ is then a coupling of $\mu$ and $\nu$, and we have
$$
W^2_2(\mu,\nu) \leq \int_{M \times M}d^2(z,x)d\gamma_\mu(z,x) =\int_{\spt(\bar \mu)}d^2(g(z),T_\mu(z))d\bar \mu(z).
$$
Therefore, 
\begin{eqnarray} \nonumber
\int_{P(M)}W_2^2(\mu,\nu)d\Omega(\mu) &\leq& \int_{P(M)} \int_{\spt(\bar \mu)}d^2(g(z),T_\mu(z))d\bar \mu(z)]d\Omega(\mu)\label{eqn: wass distance inequality}\\ \nonumber
&=&\int_{M}\int_{P(M)}d ^2(g(z),T_\mu(z))d\Omega(\mu)d\bar \mu(z)\\ 
&<&\int_{M}\int_{P(M)}d^2(z,T_\mu(z))d\Omega(\mu)d\bar \mu(z)\\ \nonumber
&=&\int_{P(M)}\int_{M}d^2(z,T_\mu(z))d\bar \mu(z)d\Omega(\mu)\\ \nonumber
&=&\int_{P(M)}W_2^2(\mu,\bar \mu)d\Omega(\mu)
\end{eqnarray}
where the second and fourth lines follow from Fubini's theorem and the strict inequality in the third line follows from \eqref{redvar} (and the fact that \eqref{redvar} is strict on a set of $\bar \mu$ positive measure).  This contradicts the fact that $\bar \mu$ is a barycenter of $\Omega$, completing the proof that $z$ is a barycenter of $\lambda_z$  for $\bar \mu$-a.e. $z$. 

To prove the second assertion, we must show that we \emph{must} have $g(z) =z$  $\bar \mu$ almost everywhere, under the additional assumption $\Omega(P_{ac}(M))>0$.  By Theorem \ref{thm: existence and uniqueness}, the barycenter $\bar \mu$ is unique, and therefore, we must have $\nu =  g_\# \bar\mu = \bar \mu$, and so we have equality throughout the preceding string of inequalities.  In particular the first line in \eqref{eqn: wass distance inequality} becomes
$$
\int_{P(M)}W_2^2(\mu,\nu)d\Omega(\mu) =\int_{P(M)} \int_{\spt(\bar \mu)}d^2(g(z),T_\mu(z))d\bar \mu(z)]d\Omega(\mu).
$$
This implies that for $\Omega$ almost every $\mu$, the plan $(g,T_\mu)_\#{\bar \mu}$ is an optimal plan bewteen $\nu =\bar \mu$ and $\mu$; as $(Id,T_\mu)_\#{\bar \mu}$ is the unique such optimal plan, by the Brenier-McCann theorem (Theorem \ref{thm: McCann}), this implies
$$
(g,T_\mu)_\#{\bar \mu}=(Id,T_\mu)_\#{\bar \mu},
$$
for $\Omega$ almost all $\mu$.  In particular, as $\Omega(P_{ac}(M)) >0$, the preceding holds for some $\mu \in P_{ac}(M)$. For such a $\mu$, this means that for $\bar \mu$ almost all $z$, there exists some $y$ such that
$$
(z,T_\mu(z)) =(g(y), T_\mu(y)).
$$ 
 Now, as $\mu$ is absolutely continuous, $T_\mu$ is injective $\bar \mu$ almost everywhere; hence, $T_\mu(z)=T_\mu(y)$ implies $z=y$.    The preceding equation then means that, $\bar \mu$ almost everywhere, we have $z=y$, and therefore,
$$
z=g(z)
$$
as desired.
\end{proof}

We now prove  the following  balance condition on the first and second order derivatives of the potential functions $u_\mu$. It will play an important role in obtaining estimates on the density of the Wasserstein barycenter  as well as the proof of our Wasserstein Jensen's  inequalities in the later part of the paper: see Theorem~\ref{thm: detcontrol}  subection~\ref{SS: upper bound}, and Sections \ref{S:regularity-infinitely many} and \ref{S:Jensen}.
 
\begin{theorem}[\bf The first and second order balance at the Wasserstein barycenter]\label{thm: bccondition}
Let $ \bar \mu$ be the Wasserstein barycenter of a measure $\Omega$ on  $P(M)$. Assume $\bar \mu$ is absolutely continuous and let $T_\mu(x) =exp_x(\nabla u_\mu(x))$ be the optimal map  pushing $\bar \mu$ forward to $\mu$, where $u_\mu$ is the associated dual potential.  Then for $\bar \mu$ almost all $x$ we have
\begin{align}\label{eq: 1st order balance}
 & \text{{\bf 1st order balance:}} 
\quad \int_{P(M)} \nabla u_\mu (x) d\Omega (\mu) =0,
\\\label{eq: 2nd order balance}
 & \text{{\bf 2nd order balance:}}
  \quad   \int_{P(M)} \nabla^2_x u_\mu (x) d\Omega (\mu)
\le 0.
\end{align}
\end{theorem}
\begin{proof}[\bf Proof of the 1st order balance \eqref{eq: 1st order balance}]
 Fix an arbitrary $x$ where the map $T_\mu$ is well defined (which holds $\bar \mu$ a.e.). Since each $u_\mu$ is a $c$-convex function, for its $c$-dual $u^c_\mu$, we have 
\begin{align*}
u_{\mu}(y)\geq -\frac{d^2(y,T_\mu(x))}{2}-u_\mu^c(T_\mu(x))  \hbox{ for any $y \in M$ and $\mu \in P(M)$},
\end{align*}

with equality when $y =x$.  Integrating against $\Omega$, we have
\begin{align}\label{eq: integral support}
 \int_{P(M)}u_{\mu}(y)d\Omega(\mu) \geq \int_{P(M)}-\frac{d^2(y,T_\mu(x))}{2}d\Omega(\mu)-\int_{P(M)}u_\mu^c(T_\mu(x))d\Omega(\mu), 
\end{align}
with equality when $y=x$.

On the other hand, by Lemma \ref{riemannianbc},  for $\bar\mu$ almost every $x$, we have that $x$ is the barycenter of  $\lambda_z =(\mu \mapsto T_\mu(z))_\#\Omega$
; that is, a minimizer of 
$$
 f_x: y \mapsto \int_{P(M)} d^2(y, T_\mu(x))d\Omega(\mu).
$$
 Therefore, the latter function $f_x$, which is semi-concave is differentiable at $x$:
 due to semi-concavity, there is $C>0$ such that the function $f_x (y)  - C\dist^2 (x, y)$ is locally geodesically concave near $x$.   Minimality at $x$ implies $f_x(y) - C\dist^2 (x, y) \ge f_x (x) -C \dist^2 (x, x)$.  Since $y \mapsto f_x(x) - C \dist^2 (x, y)$ has vanishing derivative at $x$, concavity of $f_x(y) - C\dist^2 (x, y)$ implies that  locally the function $y \mapsto f_x (y) -C \dist^2 (x, y)$  is also locally bounded from above by the constant  $f_x(x)$. This implies the differentiability of $f_x$ at $x$ as well as 
\begin{equation}\label{bcdiff}
\nabla_y\Big|_{y=x} f_x (y)=  \nabla_y \Big|_{y=x}\int_{P(M)} d^2(T_\mu(x),y)d\Omega(\mu)=0.
\end{equation}

 By assumption, $\bar \mu$ is absolutely continuous, so the Lebesgue a.e. first and second order differentiability of  the function  $y \mapsto \int_{P(M)}u_{\mu}(y)d\Omega(\mu)$ implies  $\bar \mu$-a.e.  first and second order differentiability.  Since equality holds in \eqref{eq: integral support} at $y=x$,
for $\bar \mu$-a.e. $x$, 
we have 
\begin{align}\label{eq: zero deriv integral}
\nabla_y \Big|_{y=x}\int_{P(M)}u_{\mu}(y)d\Omega(\mu) =\nabla_y \Big|_{y=x}\int_{P(M)}d^2(y,T_\mu(x))d\Omega(\mu) =0,
\end{align}
where the last equality follows from \eqref{bcdiff}.  Equation \eqref{eq: 1st order balance}  for $\bar \mu$-a.e. $x$ now follows from \eqref{eq: derivative integrals}. 
\end{proof}
 
\begin{proof}[\bf Proof of the 2nd order balance \eqref{eq: 2nd order balance}]
This follows immediately by applying Lemma~\ref{lem: a.e. 2nd diff zero} below  to 
\begin{align*}
\psi (x) =  \int_{P(M)} u_\mu (x)d \Omega (\mu),
\end{align*}
noting \eqref{eq: zero deriv integral}, and then using  \eqref{eq: 2nd deriv integral}. 
 \end{proof}
 
\begin{lemma}[\bf Vanishing derivatives imply vanishing Hessian]\label{lem: a.e. 2nd diff zero}
Let $\bar \mu$ be an absolutely continuous measure on $M$.  
Let $\psi: M \to \R$ be a Lipschitz and geodesically semi-convex function. 
Suppose for $\bar \mu$-a.e. $x$,  
\begin{align*}
 \nabla_x \psi(x) =0
\end{align*}
Then, for $\bar \mu$-a.e. $x$, 
\begin{align*}
 \nabla^2_x \psi(x) = 0.
\end{align*}
\end{lemma}
\begin{proof}
We first find a relevant set of full $\bar \mu$ measure.  Since $\bar \mu$ is absolutely continuous and   $\nabla_x \psi(x) =0$ for $\bar\mu$-a.e. $x$,   there exists a  full $\bar \mu$ measure set $S$ (i.e. $\bar \mu (M \setminus S)=0$) with the following properties at each $x \in S$:
\begin{enumerate}
 \item $\psi$ is second order differentiable (in the Alexandrov sense) at $x$,
  \item $\nabla \psi (x)  =0$;
 \end{enumerate}
Moreover, due to absolute continuity of $\bar \mu$ and the Lebesgue density theorem, 
   the set $$S' = \{ x \in S \ | \ \lim_{r \to 0} \frac{\vol(S\cap B_r (x))}{\vol(B_r(x))} =1 \}$$
  has full $\bar \mu$ measure. 
It suffices to show $\nabla^2 \psi (x) =0$ on $S'$. 

Now,  fix an arbitrary $x \in  S'$. 
 We recall that the Hessian $\nabla^2 \psi$ satisfies for each $v\in T_xM$, (see, e.g. \cite[Theorem 14.25]{V2}),
\begin{align*}
 \nabla^2 \psi (x) v & = \partial \psi (\exp_x v) - \nabla \psi (x) + o(|v|) \quad \hbox{ as $v\to 0$},\\
 & = \partial \psi (\exp_x v) + o(|v|) \quad \hbox{(since $x \in S' \subset S$)},
\end{align*}
where $\partial \psi$ denotes the subdifferential, which coincides with $\nabla \psi$ at differentiable points. 
Therefore, whenever $\exp_x v \in S$, we see
\begin{align}\label{eq: small hessian}
 \nabla^2 \psi (x)  v  = o(|v|).
\end{align}
Now, notice that since $x$ is a density point of $S$ (by definition of $S'$), for each unit vector $w \in T_xM$, $|w|=1$, there is a sequence of $v_k$ such that $\exp_x v_k \in S$, and $\frac{v_k}{|v_k|} \to w$ and  $|v_k| \to 0$ as $k \to \infty$. 
Therefore, 
\begin{align*}
 \nabla^2 \psi (x) w = \lim_{k \to \infty}  \nabla^2 \psi (x)  \frac{v_k}{|v_k|} =_{\hbox{ use \eqref{eq: small hessian}}} \lim_{k\to \infty} \frac{o(|v_k|)}{|v_k|} =0.
\end{align*}
This shows that $\nabla^2 \psi (x) =0$, completing the proof. 
\end{proof}

\subsection{Jacobian determinants of optimal maps from a Wasserstein barycenter}\label{SS: Jacobian determinant integral}
Theorem~\ref{thm: bccondition}  (both the first and second order balance properties) will be crucial in our proof of Theorem~\ref{thm: k-Jensen}. It also implies the following important property that will be used to control the density of the Wasserstein barycenter measure later.  The rough idea of the theorem below is as follows: the semi-convex functions $u_\mu$ have second derivatives which are uniformly bounded below.  The second order balance  property \eqref{eq: 2nd order balance} then implies that almost every $D^2u_\mu$ have upper bounds as well, as otherwise the resulting large positive terms in  \eqref{eq: 2nd order balance} would need to be balanced by large, negative terms, which would violate uniform semi-convexity.  This, combined with the concavity of $A \mapsto \det^{1/n} (A)$ on the set of nonnegative symmetric matrices implies estimates on the Jacobians of the optimal maps, as is made precise below.  Of course, curvature plays a key role in quantifying the semi-convexity of the potentials, and this is relfected in the appearance of the generalized distorition coefficients in \eqref{eqn: jacobian control} below.

\begin{theorem}[\bf Jacobian determinant inequality for the Wassersten barycenter]\label{thm: detcontrol}
Assume that  the Wasserstein  barycenter  $\bar \mu$ of the measure $\Omega$ on $P(M)$ is absolutely continuous.  Letting  $T_\mu$ denote the optimal map from $\bar \mu(x) $ to $\mu$, consider the measure on $M$ given by  
\begin{align}\label{eq: lambda continuous} 
\lambda_x :=\int_{P(M)} \delta_{T_\mu(x)}d\Omega(\mu),
\end{align}
which is defined with respect to a.e. $x$ (due to Lemma~\ref{lem: a.e. diff}).  Then, for $\bar \mu$-a.e. $x$, 

\begin{equation}\label{eqn: jacobian control}
1\geq \int_{P(M)}  \alpha_{\lambda_x}^{1/n}(T_\mu (x)){\det}^{1/n}DT_\mu(x)d\Omega(\mu).
\end{equation}

\end{theorem}

\begin{proof}
We have for $\Omega$-a.e. $\mu$ and a.e. $x$ (so for $\bar \mu$-a.e. $x$ for absolutely continuous $\bar \mu$), 
$$
D^2_{xx} u_{\mu}(x) + D^2_{xx}c(x,T_\mu(x)) =- D^2_{xy}c(x,T_\mu(x))DT_{\mu}(x).
$$
Rearranging, integrating against $\Omega$ and using \eqref{eq: 2nd order balance}, yields
\begin{align*}
&\int_{P(M)}- D^2_{xy}c(x,T_\mu(x))DT_{\mu}(x)d\Omega(\mu)\\
 &= \int_{P(M)}  \left[D^2_{xx} u_{\mu}(x)+D^2_{xx}c(x,T_\mu(x))\right]d\Omega(\mu)\\
 & \le  \int_{P(M)}  D^2_{xx}c(x,T_\mu(x))d\Omega(\mu) \hbox{ (by \eqref{eq: 2nd order balance})}
 \end{align*}
Note that 
 each $- D^2_{xy}c(x,T_\mu(x))DT_{\mu}(x) =D^2_{xx} u_{\mu}(x)+D^2_{xx}c(x,T_\mu(x))$ is positive semi-definite by the $c$-convexity of $u_\mu$, and  hence so is their integral,  $\int_{P(M)}- D^2_{xy}c(x,T_\mu(x))DT_{\mu}(x)d\Omega(\mu)$.
The preceding string of inequalities then  implies that  $\int_{P(M)}  D^2_{xx}c(x,T_\mu(x))d\Omega(\mu) $ must also be positive semi-definite, and so we have 
\begin{align*}
& \det\left[\int_{P(M)}- D^2_{xy}c(x,T_\mu(x))DT_{\mu}(x)d\Omega(\mu)\right]
\\
&\leq \det\left[\int_{P(M)}  D^2_{xx}c(x,T_\mu(x))d\Omega(\mu)\right].
\end{align*}
Combining Minkowski's determinant inequality with Jensen's  inequality then yields:
\begin{align*}
& \int_{P(M)}{\det}^{1/n}[- D^2_{xy}c(x,T_\mu(x))]{\det}^{1/n}[DT_{\mu}(x)]d\Omega(\mu)
\\
&\leq {\det}^{1/n}[\int_{P(M)}  D^2_{xx}c(x,T_\mu(x))d\Omega(\mu)].
\end{align*}
This yields the desired inequality, by the definitions of $\lambda_x$  (see \eqref{eq: lambda continuous}) and $\alpha_{\lambda_x}$ (see  \eqref{eq: alpha}). 
\end{proof}

\section{Absolute continuity of the Wasserstein barycenter of finitely many measures}\label{S:regularity-finitely many}
{In this section, we first establish absolute continuity of the Wasserstein barycenter $\bar \mu$ (which is itself a  measure on $M$) of  finitely many probability measures $\mu_i, i=1, ..., m$ with weights $\lambda_i \ge 0, i=1, ..., m$; see Theorem~\ref{thm: ac-for finite}.
  This  result  is interesting in its own right, but will also prove to be crucial to obtain derivative estimates on optimal maps from $\bar \mu$ to the measures $\mu_i$: see Theorem~\ref{thm: upper bound of density}.  We will use it in Section~\ref{S:regularity-infinitely many}, 
 to treat  general probability measures $\Omega$ on $P(M)$, by an approximation argument.

By definition, the Wasserstein barycenter is the metric barycenter of the probability measure $\Omega = \sum_{i=1}^m \lambda_i\delta_{\mu_i}$ on the space of probability measures $P(M)$ equipped with the Wasserstein metric $W_2$, where $\delta_{\mu_i}$ denotes the Dirac measure on $P(M)$ concentrated at $\mu_i \in P(M)$, i.e. 
for each Borel 
set (with respect to the weak-* topology)  $\Gamma \subset P(M)$, it satisfies
\begin{align*}
 \delta_{\mu} (\Gamma) = \begin{cases}
   1   & \text{if $\mu \in \Gamma $ }, \\
   0   & \text{otherwise}.
\end{cases} 
\end{align*}
In our main results, we also assume that  $\Omega(P_{ac}(M)) >0$ (that is, at least one of the  $\mu_i$ is absolutely continuous with respect to volume and $\lambda_i \neq 0$). }

We now state the main result of this section: 
\begin{theorem}[\bf Absolute continuity of the Wasserstein barycenter for finitely many measures]\label{thm: ac-for finite}
Assume $\mu_1$ is absolutely continuous with respect to volume on $M$ and let $\lambda_1 >0$.  Then the Wasserstein barycenter $\bar \mu \in P(M)$ of the measure $\Omega = \sum_{i=1}^m\lambda_i\delta_{\mu_i} \in P(P(M))$  is absolutely continuous with respect to volume on $M$.
\end{theorem}
The above result was shown in the Euclidean case $M=\R^n$ by Agueh and Carlier \cite{ac}.  However, their method exploits  the underlying  Euclidean geometry and the special algebraic structure of the multi-marginal transport in a pivotal way:  in particular, in the Euclidean setting, the unique barycenter of  the points $x_1, ..., x_m \in \R^n$ with respective weights $\lambda_1,...,\lambda_m$ is nothing but the algebraic average $\sum_{i=1}^m \lambda_i x_i$. 
To demonstrate this special nature of the Euclidean space, we provide an alternative but simpler proof to \cite{ac} of the absolute continuity of the Wasserstein barycenter in that case:
{
\begin{proof}[\bf Proof of Thereom~\ref{thm: ac-for finite} when $M=\R^n$ is the Euclidean space]
(See \cite{ac} for a different proof.)
By Theorem~\ref{thm: Kim Pass}, assertion 3, the Wasserstein barycenter measure $\bar \mu$ is given by 
$(\bar x_\lambda)_\#\gamma$, where $\gamma \in P (\Pi_{i=1}^m M) $ is the Kantorovich solution of the multi-marginal problem \eqref{K}, and $\bar x_\lambda$ is the barycenter map with weights $\lambda_i$'s: in this Euclidean setting, $\bar x_\lambda (x_1, ..., x_m) = \sum_{i=1}^m \lambda_i x_i$.
 
 Let $(x_1, .., x_m), (x_1', ..., x_m') \in \spt \gamma$, and denote $z=\bar x_\lambda (x_1, .., x_m) = \sum_{i=1}^m \lambda_i x_i$, $z'=\bar x_\lambda (x_1', .., x_m') = \sum_{i=1}^m \lambda_i x_i'$. 
Now, the support of $\gamma$, is $c$-montone:
\begin{align*}
 \sum_{i=1}^m  \lambda_i (|z- x_i|^2 + |z'-x'_i|^2) \le \sum _{i=1}^m  (\lambda_i |z'- x_i|^2 + |z'-x_i|^2);
\end{align*}
after rearranging terms, this is equivalent to
\begin{equation*}
\sum_{i \neq j} \lambda_i \lambda_j[ x_i \cdot x_j +x_i' \cdot x_j '-x_i '\cdot x_j +x_i \cdot x_j '] \geq 0.
\end{equation*}
Therefore, we have
\begin{align}\label{eq: monotone Eucl}
 |z-z'|^2&=|\sum _{i=1}^m \lambda_i(x_i-x_i')|^2\\
&=\sum_{i=1}^m \lambda_i^2 |x_i -x_i'|^2 +\sum_{i \neq j} \lambda_i \lambda_j (x_i-x_i') \cdot (x_j-x_j')\\
& \ge \sum_{i=1}^m \lambda_i^2 |x_i -x_i'|^2
\end{align}
From this we see that the inverse map $(\bar x_\lambda)^{-1}: \spt \bar \mu \to \spt \gamma$  is Lipschitz with constant $\le \frac{1}{\lambda_1}$, and therefore the composition of the inverse and the projection $\pi_1: \spt \gamma \to \spt \mu_1$ 
is Lipschitz as well, also with Lipschitz constant $\le \frac{1}{\lambda_1}$. Since $\mu_1$ is absolutely continuous, and this composition pushes $\bar \mu$ forward to $\mu_1$, this immediately implies $\bar \mu$ is absolutely continuous. Moreover, if $\mu_1 \in L^\infty (M)$, then, 
\begin{align*}
 \left \| \frac{d\bar \mu}{dx}\right \|_\infty 
\le \frac{1}{\lambda_1^n} \left \| \frac{d\mu_1}{dx}\right\|_\infty 
\end{align*}
where $\frac{d\mu}{dx}$ denotes the Radon-Nikodym derivative. This completes the proof in the Euclidean case. 
\end{proof}

Neither this proof, nor a related argument \cite{P9} on very special Riemannian manifolds (simply connected manifolds with nonpositive curvature) seem suited to handle the general Riemannian case.  In particular, an analogue of the inequality \eqref{eq: monotone Eucl} is not known in that context.

Instead, our proof is inspired by the method of Figalli and Juillet \cite{FJ}, who established absolute continuity of Wasserstein geodesics  (also known as McCann's displacement  interpolants\cite{m})   over the Heisenberg group and Alexandrov spaces, using a very nice  approximation argument.  

We devote the following subsection to the proof of Theorem~\ref{thm: ac-for finite}.

\subsection{Proof of the absolute continuity of the Wasserstein barycenter for finitely many measures}\label{SS: AC finite}

The proof requires a  few lemmata.

 Recall that  $M$ is an $n$-dimensional Riemannian manifold. To fix notation, for each Borel set $E \subseteq M$, and $m-1$ points $x_2, ..., x_m$ and the weights $\lambda_i \ge 0$, $i=1, ..., m$ (and $\sum_i \lambda_i =1$) , let
\begin{align}\label{eq: bc set}
bc_{\lambda}(E, x_2,...,x_m) : =\bigcup_{x \in E} bc_\lambda(x, x_2 ..., x_m).
\end{align}

A crucial geometric property of this set is given in the following lemma, which roughly speaking, provides a Lipschitz inverse map of  $ y \mapsto bc_{\lambda}(y, x_2, ..., x_m)$, implying bounded volume distortion.  The underlying principle behind the proof is similar to the idea behind the proof of Theorem~\ref{thm: detcontrol}:
  if the gradients of uniformly semi-concave functions (e.g $\dist^2$) are balanced in the sense of  \eqref{eq: 1st order balance}, then the second derivatives of the functions are  bounded uniformly from both above and below. For the uniform estimates in what follows, it is important that the points $x_2, ..., x_m$ are fixed.}
\begin{lemma}[\bf A Liptschitz inverse to the barycentre map]\label{lem: bcbound}
Assume $\lambda_1>0$ in \eqref{eq: bc set}. 
Then, there exists a map $$G_{\lambda; x_2, ..., x_m}: bc_{\lambda}(M, x_2,...,x_m)  \to M$$ such that for each Borel set $E$, 
\begin{enumerate}
 \item   $E=G_{\lambda; x_2, ..., x_m}(bc_{\lambda}(E, x_2,...,x_m));$
 \item $G_{\lambda; x_2, ..., x_m}$ is uniformly locally Lipschitz with a Lipschitz constant $C=C(\lambda,M)$ depending only on $\lambda =(\lambda_1,...\lambda_m)$ and $M$ (that is, not on $x_2,....x_m)$.  In particular, this implies

 $$ \vol (G_{\lambda; x_2, ..., x_m}(E \cap bc_{\lambda}(M, x_2,...,x_m) )) \le C^n \vol (E \cap bc_{\lambda}(M, x_2,...,x_m)  )$$

for any Borel set $E \subset M$.
\end{enumerate}
\end{lemma}

\begin{proof}
We claim that for each $z \in bc_\lambda (M, x_2, ..., x_m)$, there exists a unique point, which we will define to be $G_{\lambda; x_2, ..., x_m}(z)$  satisfying $$z \in bc_\lambda (G_{\lambda; x_2, ..., x_m}(z), x_2, ..., x_m).$$
 Existence of $G_{\lambda; x_2, ..., x_m}(z)$ follows from the definition of $bc_\lambda (M, x_2, ..., x_m)$.
 To see uniqueness,  we define  
$$g(z) = \frac{1}{\lambda_1}\sum_{i=2}^m \lambda_i d^2(x_i,z),$$
and recall that for each $x$, every $z\in bc_{\lambda}(x, x_2, .....x_m)$ is not in the cutlocus of any $x_i$  \cite{KP}.  Therefore, $g(z)$ is twice differentiable at each $z \in bc_\lambda (M, x_2, ..., x_m)$.  

Now, for any point $y$ such that $$z \in bc_\lambda (y, x_2, ..., x_m),$$ from the definition of $g(z)$, we have
$\nabla_{w}\Big|_{w=z} d^2(y,w) =- \nabla_z g(z)$ or, equivalently,  $y= \exp_z  \frac{1}{2} \nabla g(z).$
Therefore, $ \exp_z  \frac{1}{2} \nabla g(z)$ is the only point with the desired property, establishing uniqueness,  as well as the formula  $$G_{\lambda; x_2, ..., x_m}(z)= \exp_z  \frac{1}{2}\nabla g(z).$$
 It follows by definition that $$E=G_{\lambda; x_2, ..., x_m}\left(bc_{\lambda}(E, x_1,...,x_m)\right),$$ proving assertion 1.

To prove the second assertion,  first observe that due to minimality of $w \mapsto d^2(G_{\lambda; x_2, ..., x_m}(z),w) +g(w)$ at $w=z$ , we have 
$$\nabla^2_{w}\Big|_{w=z} d^2\left(G_{\lambda; x_2, ..., x_m}(z),w\right) +\nabla^2 g(z) \geq 0.$$

It is well known that the function $z \mapsto d^2(y,z)$ is semi-concave and satisfies the estimate $\nabla^2_z d^2(y,z) \leq C'$ for some $C'$ depending only on $M$  (see, e.g. \cite{c-ems}); it follows from the definition of $g$ that we have $\nabla^2_z g(z) \leq \frac{C'}{\lambda_1}$.  The inequality above then yields  $\nabla^2_z g(z) \geq -C'$ for each $z \in  bc_\lambda(x, x_2, ..., x_m)$.  

Now,  by the same reasoning, it follows that  for each $z \in  bc_\lambda(x, x_2, ..., x_m)$, and each $i$, we have $\nabla^2_z d^2(x_i,z) \geq \frac{-(1-\lambda_i)C'}{\lambda_i}$.  Set $K_\lambda = \max_i\frac{(1-\lambda_i) C'}{\lambda_i}$.

By continuity of $\nabla^2_z d^2(y,z)$ away from the cutlocus, and compactness of $M$, there exists an $r >0$ such that for each $y$ and each $\bar z$ with $\nabla^2_z d^2(y,\bar z ) \geq - K_{\lambda}$, we have  $\nabla^2_z d^2(y,z) \geq - K_{\lambda}-1$ for  $z \in B_r(\bar z)$.  

Therefore, for each $\bar z \in  bc_\lambda(x, x_2, ..., x_m)$, we obtain $|\nabla^2g(z)| \leq K$, for $K=\max\{\frac{C'}{\lambda_1}, K_\lambda +1\}$ on $B_r(\bar z)$;  as the exponential map is Lipschitz, we obtain that  $G_{\lambda; x_2, ..., x_m}(z)=\exp_z  \frac{1}{2}\nabla g(z)$ is Lipschitz on $B_r(\bar z)$, and therefore on $B_r(\bar z) \cap  bc_\lambda(x, x_2, ..., x_m)$, with a Lipschitz constant $C$ depending only on $\lambda$ and $M$.

  The property 
 $$\vol \left(G_{\lambda; x_2, ..., x_m}(E \cap bc_{\lambda}(M, x_2,...,x_m) )\right) \le C^n \vol \left(E \cap bc_{\lambda}(M, x_2,...,x_m)  \right)$$
now follows from standard arguments in geometric measure theory; see, for example, \cite[Propositions 12.6 and 12.12]{taylor}. 
\end{proof}

{

\begin{lemma}[\bf Absolute continuity of the Wasserstein barycenter when all but one marginal is discrete]\label{lem: ac for discrete}
Assume $\mu_1$ is a probability measure,  absolutely continuous with respect to volume on $M$ and let $\lambda_1 >0$. 
Moreover, assume that $\mu_i$, for $i=2, ..., m$ are discrete measures on $M$, i.e., each $\mu_i$, $i\ge 2$ is of the form $\mu_i = \frac{1}{N_i}\sum_{j=1}^{N_i} \delta_{x^j_i}$ with $x^j_i\in M$. 
Let $\bar \mu$ be the unique Wasserstein barycenter measure of the measures $\mu_i$ with weights $\lambda_i$, $i=1, .., m$. 

Then there is a finite collection of points $\{(x_2^j,....,x_m^j)\}_{j=1}^L \subseteq \Pi_{i=2}^m M$ such that, for any Borel $E \subset M$, we have
$$
 \bar \mu (E) =\sum_{j=1}^L  \mu_1 (G_j (bc_j \cap E)) 
$$
where, using the notation from the last lemma,
\begin{align*}
G_j &: =  G_{\lambda; x_2^{j},...x^{j}_m}\\
bc_j &: =  \spt \bar \mu \cap bc_{\lambda}(M, x_2^{j},...x^{j}_m )
\end{align*}
\end{lemma}
\begin{proof}

  From Theorem~\ref{thm: Kim Pass}, there exists a unique multi-marginal optimal plan, say $\gamma$, for the $m$ measures $\mu_1, \mu_2, ..., \mu_m$; moreover,  the Wasserstein barycenter measure  $\bar \mu$ is given by 
$\bar \mu=\bar x_{\lambda\#} \gamma$.

Since the measures $\mu_i$, $i=2, ..., m$ are discrete measures on $M$, the support $\spt \gamma$  is an almost everywhere disjoint union of sets $W_j$, $j=1, ..., L$ satisfying
\begin{align*}
 \pi_2 \times ... \times \pi_m \left( W_j \right) = \{(x_2^{j},...x^{j}_m)\}.
\end{align*}
for some $x_i^{j} \in M$, in the support of the discrete measure $\mu_i$ for $i=2, ..., m$.

 From Theorem~\ref{thm: Kim Pass}, the maps $\pi_1: \spt \gamma \to \spt \mu_1$ and 
$\bar x_\lambda: \spt \gamma \to \spt \bar \mu$   are one-to-one and onto ($\gamma$-a.e.); thus, the correspondence $G_j$, restricted to $\spt \bar \mu$ is one-to-one and onto $\pi_1(W_j)$  ($\mu_1$ a.e.) as well. Note that this is the only place where we use optimality of $\gamma$ (or the optimality of $\bar \mu$).
The above bijections  now give a partition of $\spt \mu_1= \cup_{j=1}^{L} \pi_1(W_j)$, and 
 a partition $\spt \bar\mu = \cup_{j=1}^L bc_j$, up to sets of $\mu_1$ measure $0$ and sets of $\bar \mu$ measure $0$, respectively.

Therefore, from $\bar \mu = \bar x_{\lambda\#} \gamma$ and $\pi_{1\#}\gamma=\mu_1$,  we see, for any Borel set $E \subset M$,
\begin{align}\label{eq: push forward}
  \bar \mu (E) & = \mu_1 (\cup_{j=1}^L G_j (bc_j \cap E))\\\nonumber
  &= \sum_{j=1}^L  \mu_1 (G_j (bc_j \cap E))\\
  &  \qquad \hbox{(since the $G_j(bc_j)$'s are disjoint  modulo a set of $\mu_1$ measure $0$.})
\end{align}
establishing the claim.
\end{proof}
}

We also need the following continuity of  the Wasserstein barycenter under weak-* convergence.
\begin{lemma}[\bf Continuity of Wasserstein barycenter in the weak-* topology]\label{weakcon}
Suppose $\mu_i^N$ converges in the weak-* topology to $\mu_i$ for each $i$, as $N \rightarrow \infty$.  
Then any weakly-* convergent subsequence of  Wasserstein barycenter measure $\bar \mu^N$ of the measures  $\mu_1^N, ..., \mu_m^N$ with weights $\lambda_1,...\lambda_m$ converges in the weak-* topology to a Wasserstein barycenter measure $\bar \mu$ of the measures $\mu_1, ..., \mu_m$, with the same weights.
\end{lemma}
\begin{proof}
For each measure $\nu$, we have by definition

\begin{equation*}
\sum_{i=1}^m \lambda_iW_2^2(\mu_i^N,\bar \mu ^N) \leq\sum_{i=1}^m \lambda_iW_2^2(\mu_i^N,\nu) 
\end{equation*}
Now, passing to  any weak-* convergent subsequence $\bar \mu^N \rightarrow \bar \mu$ and taking the limit in the preceding inequality and  using continuity of $W_2$ with respect to the weak* topology, 
we obtain

\begin{equation*}
\sum_{i=1}^m \lambda_iW_2^2(\mu_i,\bar \mu ) \leq\sum_{i=1}^m \lambda_iW_2^2(\mu_i,\nu).
\end{equation*}
As $\nu$ was arbitrary, this implies that $\bar \mu$ is by definition a barycenter of the $\mu_i$'s, as desired
\end{proof}

We now prove  the main theorem of this section:
\begin{proof}[\bf Proof of Theorem~\ref{thm: ac-for finite}]
{We approximate each $\mu_2,...\mu_m$ in the weak-* topology by linear combinations $\mu_2^N,...\mu_m^N$ of $N$ Dirac masses. 
Note that by Lemma \ref{weakcon}, the Wasserstein barycenters $\bar \mu^N$ converges in the weak-* topology to $\bar \mu$, the Wasserstein barycenter measure of the original measures $\mu_1, ... , \mu_m$,  which is unique due to Theorem~\ref{thm: existence and uniqueness} and the assumption that $\mu_1$ is absolutely continuous.

We establish absolute continuity of $\bar \mu$  by contradiction.  If $\bar \mu $ is not absolute continuous, then there is a Lebesgue measure zero set, say $S$, such that $\bar \mu (S) \ge \delta $, for some positive $\delta >0$.  We can choose  small (open) neighbourhoods of $S$, say, $U_k$ with $\vol (U_k) \le 2^{-k}$; as $S \subseteq U_k$, we have $\bar \mu (U_k) \ge 
\delta$. Now, due to the weak-* convergence, for a large $N_k$, we have $\bar \mu^{N_k} (U_k) \ge \delta/2$. 
Then, by Lemma \ref{lem: ac for discrete}, we have
\begin{align*}
  \bar \mu^{N_k} (U_k)  & =  \mu_1 (\cup_{j=1}^{L_{N_k}}G_j (bc_j \cap U_k)).
    \end{align*}
  But, note that from Lemma~\ref{lem: bcbound} assertion 2,  
\begin{align*}
 \vol(\cup_{j=1}^{L_{N_k}}G_j (bc_j \cap U_k)) \le C^n\vol (U_k).
\end{align*}
 This implies there is a sequence of Borel sets $V_k$ (i.e. $V_k =  (\cup_{j=1}^{L_{N_k}}G^k_j (bc_j \cap U_k))$)  with $\vol (V_k) \lesssim 2^{-k}$, and
\begin{align*}
\mu_1 (V_k) \ge \delta/2.
 \end{align*}

 This is a contradiction and  completes the proof, since 
 the absolute continuity of $\mu_1$ (with respect to $\vol$) is equivalent to the following property: 
 for every  $\epsilon >0$, there is  $\eta>0$ such that $\mu (A) \le \epsilon$ for all Borel sets $A$ with $\vol(A) \le \eta$.

 }
\end{proof}

\subsection{Upper bound of the density of the Wasserstein barycenter}\label{SS: upper bound}
In this subsection, we prove that the density of the barycenter can be controlled by the densities of the marginals (see Theorem~\ref{thm: upper bound of density} below).  Aside from being interesting in its own right, this result will be used in the proof of Theorem~\ref{thm: ac for general} on absolute continuity of the Wasserstein barycenter of a general measure on $P(M)$.

 Before presenting the main theorem of the section, we introduce the following notation:
\begin{definition}[\bf The set $\mathcal{A}_L$]\label{def: A L}
  For $0 < L < \infty$, let $\mathcal{A}_L$ be the set of Borel probability measures on $M$, absolutely continuous with respect to volume, whose densities have $L^{\infty}$ norm less than or equal to $L$. 
\end{definition}
Note that, since the bound on the $L^\infty$ norm is preserved under weak-* convergence, $\mathcal{A}_L$ is a weakly-* closed, and thus Borel measurable, subset of $P(M)$.
Here is the main theorem of the section:

\begin{theorem}[\bf Upper bound on the density of the Wasserstein barycenter]\label{thm: upper bound of density}
Fix $L>0$.   Let $\bar \mu^N =\bar f^Ndvol$ be the barycenter of the measures $\mu_i$, $i=1,2,...,N$, with weights $\lambda_i$ and assume that at least some of the $\mu_i=g_idvol$  belong to $\mathcal{A}_L$ (note that this condition ensures uniqueness and absolute continuity of the barycenter $\bar \mu^N$, by Theorems \ref{thm: existence and uniqueness} and \ref{thm: ac-for finite}, respectively).  Then we have
\begin{align*}
C \| \bar f^N \|_\infty \le [\sum_{\mu_i \in \mathcal{A}_L} \lambda_i]^{-n} \sup_{\mu_i \in \mathcal{A}_L} \|g_i\|_\infty . 
\end{align*}
where $C$ is the constant from \eqref{eq: Ric -K case}.
\end{theorem}
\begin{proof}

Let $T_i^N$ be the optimal maps from $\bar \mu^N$ to $\mu_i$. Note that, in this setting, Theorem \ref{thm: detcontrol} reduces to
\begin{align}\label{eq: DT i ineq}
 1 \geq \sum_{i=1}^N \lambda_i \alpha_{\lambda_x}(T^N_i(x))^{1/n} {\det}^{1/n} DT^N_i(x)
\end{align}
for almost every $x$.

The remainder of the proof follows an argument in \cite{P5}. 
For a.e. $x \in \spt \bar \mu^N$, we have the following Jacobian determinant equations
\begin{align*}
 \det DT_i^N (x) = \frac{\bar f^N (x) }{g_i (T_i ^N (x) )} . 
\end{align*}
Using this and \eqref{eq: Ric -K case}, we rearrange \eqref{eq: DT i ineq} to get 
\begin{align*}
 C \bar f^N (x) \le  \left[ \sum_{\mu_i \in \mathcal{A}_L}  \frac{\lambda_i}{g_i^{1/n} (T_i^N(x))}\right]^{-n}
\end{align*}
for $C=C(\diam(M), K,n)$ in \eqref{eq: Ric -K case}. 
Applying convexity of $0< t \mapsto t^{-n}$ to this, we see
\begin{align*}
  C \bar f^N (x) \le [\sum_{\mu_i \in \mathcal{A}_L} \lambda_i ]^{-n-1}\sum_{\mu_i \in \mathcal{A}_L} \lambda_i 
  g_i (T_i^N(x)).
\end{align*}
In particular, 
\begin{align*}
C \| \bar f^N \|_\infty \le [\sum_{\mu_i \in \mathcal{A}_L} \lambda_i]^{-n} \sup_{\mu_i \in \mathcal{A}_L} \|g_i\|_\infty . 
\end{align*}
\end{proof}

\section{Absolute continuity of the Wasserstein barycenters of general distributions}\label{S:regularity-infinitely many}
In this section, we establish the absolute continuity of the {Wasserstein barycenter} of a general measure $\Omega$ on $P(M)$  under a reasonable assumption:

\begin{theorem}[\bf Absolute continuity of barycenters of  general measures on $P(M)$]\label{thm: ac for general}
Let $\Omega$ be a  probability measure on Wasserstein space $P(M)$ over   an $n$-dimensional compact Riemannian manifold $M$. Assume that $Ric_M \ge K$ for $K \in \R$.
Assume $\Omega(\mathcal{A}_L) >0$.
  Then, the Wasserstein barycenter measure  $\bar \mu$ of $\Omega$ is absolutely continuous on $M$ with density $\bar f$
 satisfying   
\begin{align*}
 \| \bar f\|_\infty \le \frac{L}{C\Omega(\mathcal{A}_L) ^n }
\end{align*}
where $C=C(M)$ is the constant given in \eqref{eq: Ric -K case}.
  \end{theorem}

The proof is by approximation; as $P(M)$ is itself a complete separable metric space, we can approximate the measure $\Omega$ by finitely supported measures $\Omega^N$, which have absolutely continuous Wasserstein barycentres as shown in Theorem~\ref{thm: ac-for finite} (see Lemma \ref{lem: polish} below).  To pass to the limit, we require the uniform estimates  from Theorem~\ref{thm: upper bound of density}, which then require the technical but  reasonable  hypothesis $\Omega(\mathcal{A}_L) >0$.

We will need the following  topological lemma (a proof can be found, for example, in \cite[Theorem 6.18]{V2} ).
\begin{lemma}[{\bf Approximation by Dirac deltas; see, e.g.  \cite[Theorem 6.18]{V2}.}]\label{lem: polish} 
  For any complete separable metric space $X$, its Wasserstein space $P(X)$ is also a complete separable metric space. Moreover, for each Borel probability measure $\nu$ on $X$, there exists a sequence $\nu^N=\sum_{i=1}^N\lambda_i \delta_{x_i}$ of  finitely supported  probability measures on $X$ converging in the weak-* topology to $\nu$.
  \end{lemma}

\begin{proof}[\bf Proof of Theorem~\ref{thm: ac for general}]

 Decompose  $$\Omega =\Omega(\mathcal{A}_L)\,  \Omega_{\mathcal{A}_L} + (1-\Omega(\mathcal{A}_L))\, \Omega_{P(M) \setminus \mathcal{A}_L},$$ where $\Omega_{\mathcal{A}_L}$ and $\Omega_{P(M) \setminus \mathcal{A}_L}$ 
denote the probability measures obtained by restricting $\Omega$ to $\mathcal{A}_L$ and $P(M)\setminus \mathcal{A}_L$, respectively, then normalizing. 

It is clear that $\mathcal{A}_L$ is closed in the weak-* topology of $P(M)$, and therefore is itself a complete separable metric space, and so applying  Lemma~\ref{lem: polish} to $X =\mathcal{A}_L$ yields a sequence $\Omega_{\mathcal{A}_L}^N$ of  finitely supported   probability measures on $\mathcal{A}_L$ converging in the weak-* topology  (on $P(\mathcal{A}_L)$) to $\Omega_{\mathcal{A}_L}$.  Similarly, applying Lemma~\ref{lem: polish} to $X =P(M)$ we can find a sequence $\Omega_{P(M) \setminus \mathcal{A}_L}^N$ of finitely supported probability measures on $P(M)$ converging  in the   weak-* topology on $P(P(M))$ to $\Omega_{P(M) \setminus \mathcal{A}_L}$ (note however, that the $\Omega_{P(M) \setminus \mathcal{A}_L}^N$  needs not be supported on the open set $\Omega_{P(M) \setminus \mathcal{A}_L}$).
 
We then have that the finitely supported measures $$\Omega^N :=\Omega(\mathcal{A}_L)\,  \Omega_{\mathcal{A}_L}^N + (1-\Omega(\mathcal{A}_L))\, \Omega_{P(M) \setminus \mathcal{A}_L}^N$$converge weakly-* to $\Omega$, and by construction, for each  $N$ we have $\Omega^N(\mathcal{A}_L)  \geq \Omega(\mathcal{A}_L) >0$.  We will denote  $\Omega^N  =\sum_{i=1}^N\lambda_i \delta_{\mu_i}$ .
 Denote by $\bar \mu^N \in P(M)$  the  Wasserstein barycenter of $\Omega^N$; $\bar \mu^N$ is unique and  absolutely continuous as a result of Theorems ~\ref{thm: ac-for finite} and ~\ref{thm: upper bound of density}, and the fact that  $\Omega^N(\mathcal{A}_L) >0$.

Let $T_i^N$ be the optimal maps from $\bar \mu^N$ to $\mu_i$. Let $\bar f^N$, $g_i^N$ be the density functions for the absolutely continuous measures
$\bar \mu^N$ and $\mu_i \in \mathcal{A}_L$, respectively.
From Theorem~\ref{thm: upper bound of density}, 
\begin{align*}
C \| \bar f^N \|_\infty &\le [\sum_{\mu_i \in \mathcal{A}_L} \lambda_i]^{-n} \sup_{\mu_i \in \mathcal{A}_L} \|g_i\|_\infty \\
&\le\frac{ L}{\Omega(\mathcal{A}_L)^n} . 
\end{align*}

Now, for any open $A \subseteq M$, we can pass to the weak-*  limit, as $\bar \mu(A) \leq \liminf_{N \rightarrow \infty}\bar \mu^N(A)$, to obtain,
\begin{align*}
 \bar \mu (A)  & \le \frac{1}{C} \liminf_{N\to \infty}  \int_A \bar f^N (x) dx \\
& \le\frac{ L}{C\Omega(\mathcal{A}_L)^n} \vol (A).
\end{align*}
For a non open set $A \subseteq M$, we get the same inequality using an approximation. This inequality is equivalent to the desired bound on $||\bar f||_\infty$, and so the proof is complete. 
\end{proof}
Once the absolute continuity of the barycenter is established, one can use Theorem~\ref{thm: detcontrol} to obtain refined estimates on its density in terms of the generalized volume distortion coefficents.

\begin{corollary}[\bf Density estimates with volume distortion]\label{co: density volume distortion}
Assume the conditions in Theorem \ref{thm: ac for general}, and that  $\Omega$-a.e $\mu$ is absolutely continuous. Denote the density of the barycentre $\bar \mu$ by $\bar f$ , the density of the measure $\mu$ by $f_\mu$, and the optimal map between $\bar \mu$ and $\mu$ by $T_\mu$.  Then, for almost all $x$, we have
$$
\bar f(x) \leq  [\int_{P(M)}\frac{\alpha_{\lambda_x}^{1/n}(T_\mu (x))}{(f_\mu(T_{\mu}(x)))^{1/n}}d\Omega(\mu)]^{-n}
$$
\end{corollary}
\begin{proof}
From Lemma~\ref{lem: a.e. diff}, for a.e. $x$, $T_\mu$ is differentiable at $x$ for $\Omega$-a.e $\mu$ with the change of variable formula,  
$$
\bar f(x)=f_\mu(T_\mu(x))\det DT_\mu(x).
$$ 
Moreover, wherever the above equation holds and $\bar f(x)$ is non zero, $f_\mu(T_\mu(x))$ is clearly nonzero as well. 
Multiply the change of variables equation  by the coefficients  $\alpha_{\lambda_x}(T_\mu (x))/f_\mu(T_\mu(x))$, take the resulting expression to the $1/n$ power and integrate to obtain
\begin{equation*}
\bar f^{1/n}(x)\int_{P(M)}\frac{\alpha_{\lambda_x}^{1/n}(T_\mu (x))}{f_\mu^{1/n}(T_\mu(x))}d\Omega(\mu)=\int_{P(M)} \alpha_{\lambda_x}^{1/n}(T_\mu (x)){\det}^{1/n}DT_\mu(x) d\Omega(\mu)
\end{equation*}
The righthand side is less than $1$ by Theorem~\ref{thm: detcontrol}, completing the proof.
\end{proof}

Using this and Proposition~\ref{prop: alpha bound}, we get an analogue of \cite[Corollary 19.5]{V2}:
\begin{corollary}[\bf Density estimates under $Ric\ge 0$]\label{co: density under Ric}
Assume the conditions in Theorem \ref{thm: ac for general}, that $\Omega$-a.e $\mu$ is absolutely continuous and  that $Ric \geq 0$.  Then 
$||\bar f||_{L^{\infty}} \leq \left \| ||f_\mu||_{L^{\infty}(M)}\right\|_{L^{\infty}(\Omega)} $,  where $\| \cdot \|_{L^\infty(\Omega)}$ is the $L^\infty$-norm on $P(M)$ with respect to the measure $\Omega \in P(P(M))$. 
\end{corollary}

\section{Convexity over Wasserstein barycenters}\label{S:Jensen}

The goal of this section is to establish a geometric version  of Jensen's  inequality  for measures $\Omega$ on $P(M)$ with respect to  a class of functionals on  $P(M)$ that are studied extensively in the literature. In fact, we establish two forms of Jensen's inequality on $P(M)$, which we will call {\em Wasserstein Jensen's inequalities}.   See Theorems \ref{thm: k-Jensen} and \ref{thm: distortedconvexity} below for the results and recall the historical discussion of Jensen's inequality in Section~\ref{sec: intro}.

Our proofs below use in a crucial way the absolute continuity of  the Wasserstein barycenter $\bar \mu$ of $\Omega \in P(P(M))$, as   established in Theorem~\ref{thm: ac for general}. This fact guarantees the existence and uniqueness of optimal maps $T_\mu$ from $\bar \mu$  to each $\mu \in \spt \Omega$, via the Brenier-McCann theorem (Theorem \ref{thm: McCann}), and therefore allows us to use  the first and second order balance conditions (with respect to $\Omega$) of $T_\mu$ given in Theorem~\ref{thm: bccondition} as well as the associated Jacobian determinant equation for $T_\mu$.  Moreover, many of the functionals we consider are defined only on the subset $P_{ac}(M)\subseteq P(M)$ of absolutely continuous measures.

Finally, let us note that   our Wasserstein Jensen's  inequalities on $P(M)$ with a curved underlying  space $M$ cannot be easily established by adapting the techniques from the Euclidean case, $M=\R^n$.
 In the Euclidean case, i.e. in $P(\R^n)$, the result was proven by Agueh-Carlier\cite{ac} (when the support of the measure $\Omega$ on $P(\R^n)$ is  finite), using so-called generalized geodesics, which amount to exponentiation of the linear interpolation of $c$-convex functions.  A key ingredient in their proof is the fact that, in $\R^n$ with $c=|x-y|^2/2$, the linear interpolant of two $c$-convex functions is again $c$-convex.  However, this method is restrictive, since $c$-convexity of functions is not preserved under linear interpolation on more general spaces.  In fact, as shown in \cite{fkm}, this property holds only on  so-called nonnegatively cross curved spaces, which are spaces  satisfying a certain fourth order condition on the metric;  this fourth order condition is a strenthened variant of the celebrated Ma-Trudinger-Wang condition \cite{mtw, tw1} arising in the regularity theory of optimal transport maps (see also \cite{loeper, km}).    Nonnegatively cross curved spaces 
  include, for example, Euclidean space, the round sphere  \cite{loeper2} and its small perturbations in two dimensions \cite{fr, Delanoe-Ge2}, products and quotients of these \cite{km2}, and some other symmetric spaces \cite{Delanoe-Rouviere}, but exclude many other spaces, including any manifold with negative sectional curvature anywhere \cite{loeper}, as well as \emph{some} manifolds with everywhere nonnegative sectional curvature\cite{kim}.

\subsection{Functionals on $P(M)$
}\label{ss: functionals}
 In this subsection, we recall certain classes of functionals on $P(M)$, which are widely studied in the optimal transport literature, and their basic properties. These are entropy type functionals (also known as internal energy functionals), potential energy functionals, and interaction energy functionals (see e.g., \cite{m, ags, V, V2}).   Due to the notions of displacement interpolation and displacement convexity, uncovered by McCann \cite{m} and discussed here in the introduction, optimal transport plays an important role to the study of these functionals. Our notation below is mostly borrowed from the book of Villani \cite{V2}.

 We start with entropy type functionals, which, among other uses, model internal energy of gases:
 \begin{definition}[\bf Entropy type functionals; see, e.g. \cite{ags, V, V2}.]\label{def: functional U}
 Let $U: [0, \infty) \to \R$ be a differentiable function, and $\nu$ a probability measure on $M$  with  $d\nu(x)= e^{-V(x)}d\vol(x)$. 
Define the functional $\mathcal{U}_\nu: P_{ac}(M) \to \R$  by $$\mathcal{U}_\nu[\mu] :=\int_{M}U(f^\nu_{\mu}(x))d\nu (x),$$
where $f^\nu_\mu$ is the density of $\mu$ with respect to $\nu$, i.e. 
$d\mu(x) = f^\nu_\mu (d) d\nu(x)$.  For measures $\mu$ which are not absolutely continuous with respect to $\nu$, we define $\mathcal{U}_{\nu} (\mu) =+\infty$. 
Also, for $r \in [0,\infty)$ and a semi-convex function $u:M \rightarrow \mathbb{R}$, we define
\begin{align}\label{eq: p r}
 p(r) = rU'(r)-U(r)\\\nonumber
 Lu = \Delta u - \nabla V \cdot \nabla u.
\end{align}
Note that  $Lu$ is well defined wherever $u$ is twice differentiable;   for semi-convex functions $u$,  the Laplacian  $\Delta$ in $L$ should be understood to be the Alexandrov Laplacian (defined Lebesgue a.e.), or equivalently the absolutely continuous part of the distributional Laplacian.
\end{definition}

  Before giving further conditions on these functionals,
let us first recall McCann's displacement interpolation (we consider only the special form where the first endpoint measure $\mu_0$ is absolutely continuous):
\begin{definition}[\bf Displacement interpolation \cite{m}]\label{def: dis inter}
 For any    absolutely continuous measure $\mu_0 \in P_{ac}(M)$, any other measure $\mu\in P(M)$, and $0 \leq t \leq 1$, we call the measure $\mu_t = (T_t)_\# \mu_0$, $t \in [0,1]$ {\em the displacement interpolant} between $\mu_0$ and $\mu$. Here, the map $T_t$ is the optimal transport map given by  $T_t (x) = \exp_x t\nabla u_\mu(x)$ for $\mu_0$-a.e. $x$, for a $c$-convex function $u_\mu$. 
\end{definition}
Part of the importance of this notion is that the curve $t \mapsto \mu_t$ coincides with the minimal geodesic segment in $(P(M), W_2)$ between the two measures $\mu_0$ and $\mu$ \cite{m};  see \cite{ags, V, V2} for an extensive review.

There are two main conditions we will impose on the entropy type functionals $\mathcal{U}_\nu$.

\begin{condition}[\bf Conditions on entropy type functionals]\label{as: properties of U}
Our conditions on $\mathcal{U}_\nu$ are:
 \begin{enumerate}
\item $U(0)=0$ and $r\mapsto p(r)$ is a continuous function, and $p \ge 0$. 
\item   (see \cite[equation (23.29)]{V2}) For the displacement interpolation $\mu_t$ as given in Definition~\ref{def: dis inter} for absolutely continuous measure $\mu_0$ with $d\mu_0 (x) = f^\nu_{\mu_0} (x) d\nu(x)$,  
  \begin{align}\label{eq: U derivative}
\liminf_{t\to 0+}  \frac{\mathcal{U}_\nu (\mu_t) - \mathcal{U}_\nu(\mu_0)}{t} = - \int_M p(f^\nu_{\mu_0}(x)) L u_\mu (x) d\nu(x). 
\end{align}
 (This holds for a large class of functionals $U_\nu$, see, for example, the \cite[proof of (23.29), page 667]{V2}.) 
\end{enumerate}

\end{condition}

Other important types of functionals are given below:
\begin{definition}[\bf Potential energy functionals and interaction energy functionals; see e.g., \cite{ags, V, V2}]\label{def: potential interaction}
For  $\tilde V: M \to \R$ and $\tilde W: M \times M \to \R$,  the corresponding 
 {\em  potential energy functional} and {\em interaction energy functional}   
 $\mathcal{V}, \mathcal{W}: P(M) \to \R\cup\{+\infty\}$ are defined, respectively, as  
  \begin{align*}
 \mu \mapsto \mathcal{V} (\mu)  &= \int_M \tilde V (x) d\mu(x), \\
  \mu \mapsto \mathcal{W} (\mu) & = \int_M \int_M \tilde W(x, y) d\mu(x) d\mu(y).
\end{align*}
\end{definition}

\subsection{First order balance for functionals on $P(M)$ at Wasserstein barycenters}\label{ss: first order functionals}

Our proof of the Wasserstein Jensen's inequalities will rely on the following lemma, which we believe holds independent interest. Our proof of the lemma exploits the first and second order balance conditions (with respect to $\Omega$) given in Theorem~\ref{thm: bccondition}.

\begin{lemma}[\bf First order balance for entropy type functionals at Wasserstein barycenters]\label{lem: lim inf U} Let $\mathcal{U}_\nu$ be an entropy type functional given in Definition~\ref{def: functional U}. Assume that $\mathcal{U}_\nu$ satisfies  Condition~\ref{as: properties of U}. 
Let $\bar\mu$ be the Wasserstein barycenter of $\Omega$, which is assumed to be absolutely continuous.  Use the notation in Condition~\ref{as: properties of U}, letting $\mu_0=\bar \mu$. Then,
\begin{align*}
\int_{P(M)}\liminf_{t \to 0+}  \frac{\mathcal{U}_\nu[\mu_t] - \mathcal{U}_\nu[\bar \mu]}{t}d\Omega(\mu)  \ge 0.
\end{align*}
\end{lemma}

\begin{proof}
 Apply part 2, i.e.  \eqref{eq: U derivative}, of Condition~\ref{as: properties of U}  to get 
\begin{align*}
 & \liminf_{t \to 0+}  \frac{\mathcal{U}_\nu[\mu_t] - \mathcal{U}_\nu[\bar \mu]}{t} 
 {\red =}
   - \int_M p(f^\nu_{\bar \mu}(x)) L u_\mu (x) d\nu(x) 
 \\
 & = 
 -  \int_M p(f^\nu_{\bar \mu}(x)) \Delta u_\mu (x)d\nu(x) 
   +   \int_M p(f^\nu_{\bar \mu}(x)) \nabla V (x)\cdot \nabla u_\mu (x) d\nu(x).  
   \end{align*}
 Due to Lemma~\ref{lem: a.e. diff}  (which applies due to the absolute continuity of $\bar \mu$)  $u_\mu(x)$ is twice differentiable in $x$ for $\Omega$-a.e. $\mu$, for a fixed $x$, $\nu$-a.e. Therefore, we can  integrate the above expression over $P(M)$ with respect to the probability measure  $\Omega$. 
 Now, observe that due to semi-convexity of $u_\mu$ as well as the continuity and nonnegativity of $p$, the negative part of the integrand involving $\Delta u_\mu$ is uniformly bounded (recall that $M$ is compact). Similarly, the integrand involving $\nabla u_\mu$ is uniformly bounded. Therefore, we can use the Fubini-Tonelli theorem to exchange the order of integrals $\int_M d\nu$ and $\int_{P(M)}d\Omega$, and get 
\begin{align}\label{eq: Jensen's proof Lap}
& -  \int_M p(f^\nu_{\bar \mu}(x)) \Delta u_\mu (x)d\nu(x)     \\\nonumber 
  & = 
  -   \int_M p(f^\nu_{\bar \mu}(x))\left[ \int_{P(M)} \Delta u_\mu (x)d\Omega(\mu)\right]  d\nu(x) 
  \end{align}
  and 
   \begin{align}\label{eq: Jensen's proof grad}
   &\int_M p(f^\nu_{\bar \mu}(x)) \nabla V (x)\cdot \nabla u_\mu (x) d\nu(x) \\\nonumber
 & \quad =  \int_M p(f^\nu_{\bar \mu}(x))  \nabla V (x)  \cdot \left[\int_{P(M)}  \nabla u_\mu (x)  d\Omega (\mu)\right]d\nu(x) 
\end{align}

Finally, apply  the  first and second order balance properties \eqref{eq: 1st order balance} and \eqref{eq: 2nd order balance} to the integrals of $\Delta u_\mu$ and $\nabla u_\mu$ with respect to $\Omega$, and 
use $p (f_\mu) \ge 0$. We see that righthand side of  \eqref{eq: Jensen's proof Lap} is nonnegative while the righthand side of  \eqref{eq: Jensen's proof grad} vanishes. This completes the proof.
\end{proof}

\begin{remark}[Remark on the proof of Lemma~\ref{lem: lim inf U}]
 Note that we could prove the above lemma  using only the first order balance condition \eqref{eq: 1st order balance} (that is, without relying on the second order balance condition~\eqref{eq: 2nd order balance}),   if we assumed $p(f_{\bar \mu}^\nu)\in W^{1,1}_{loc}(M)$. In that case, we can use the result in \cite[(23.42)]{V2}:
\begin{align*}
 -\int_M p( f_{\bar\mu}^\nu (x)) L u_\mu (x) d\nu(x) \ge \int_M \nabla u_\mu(x)  \cdot \nabla p(f_{\bar \mu(x)}^\nu) d\nu (x) .
\end{align*}
However, it is not known at present whether $p(f_{\bar \mu}^\nu)\in W^{1,1}_{loc}(M)$, even under the assumption that  $\Omega$-a.e. $\mu$ is smooth.   It is almost certainly not true in general under our much weaker sufficient condition for absolute continuity of $\bar \mu$, $\Omega(\mathcal{A}_L) > 0$. 
 
\end{remark}

For potential energy and interaction energy functionals, a similar result can be proved, in an easier way, using only the first order balance condition \eqref{eq: 1st order balance}:

\begin{lemma}[\bf First order balance for potential energy and interaction energy functionals at Wasserstein barycenters]\label{lem: lim inf V W}
 For the potential energy and interaction energy functionals $\mathcal{V}$ and $\mathcal{W}$ given in Definition~\ref{def: potential interaction},   assume that the functions $\tilde V: M \to \R$, $\tilde W: M\times M \to \R$ are  Lipschitz.
 
 Let $\bar\mu$ be the  Wasserstein barycenter of $\Omega$, and assume it is absolutely continuous. Use the notation in Definition~\ref{def: dis inter}, letting $\mu_0=\bar \mu$. Then, 
\begin{align*}
 &\int_{P(M)}\liminf_{t \to 0+}  \frac{\mathcal{V}[\mu_t] - \mathcal{V}[\bar \mu]}{t} d\Omega(\mu)
 = 0 ;\\
& \int_{P(M)}  \liminf_{t \to 0+}  \frac{\mathcal{W}[\mu_t] - \mathcal{W}[\bar \mu]}{t} d\Omega(\mu)
=0 .
\end{align*}

\end{lemma}
\begin{proof}
We recall the well known fact that (see, for example, \cite{ags, V})
\begin{align*}
 &\liminf_{t \to 0+}  \frac{\mathcal{V}[\mu_t] - \mathcal{V}[\bar \mu]}{t} 
 = \int_M  \nabla \tilde V (x) \cdot \nabla u_{\mu} (x)  d\bar \mu (x);\\
&  \liminf_{t \to 0+}  \frac{\mathcal{W}[\mu_t] - \mathcal{W}[\bar \mu]}{t} 
\\
& = \int  [\nabla_x \tilde W (x, y)\cdot \nabla u_\mu (x) + \nabla_y W (x, y)\cdot \nabla u_\mu (y)]    d\bar \mu (x) d\bar \mu(y).
\end{align*}
In both cases, after integrating against $\Omega$ and using Fubini's theorem, the righthand sides vanish due to the first order balance condition  \eqref{eq: 1st order balance}. 
\end{proof}

\subsection{ $k$-displacement convex functionals}\label{SS: k-dis convex}

In this subsection, we first recall the notion of  $k$-displacement convexity \cite{m}, and then some key examples of $k$-displacement convex functionals  in the Riemannian setting. Although potential energy and interaction energy functionals also have important applications, here we will only consider, for simplicity, examples of entropy type functionals; these have been key tools in applications of optimal transport theory to Riemannian geometry; see \cite{V2} for an extensive review. 

\begin{definition}[\bf $k$-displacement convexity \cite{m}]\label{def: k-dis convex}
A functional $\mathcal{F}: {\rm dom}(\mathcal{F})\subset P(M) \to \R \cup \{ +\infty\}$ is said to be  {\em $k$-displacement convex} for $k \in \R$, if  
for each displacement interpolation $\mu_t$ between endpoint measures $\mu_0$ and $\mu_1$ (Definition~\ref{def: dis inter}), we have 
\begin{align*}
 \mathcal{F} (\mu_t ) \le (1-t) \, \mathcal{F} (\mu_0)  + t\, \mathcal{F} (\mu_1)  - \frac{k}{2} t (1-t) W_2^2 (\mu_0, \mu_1) .
\end{align*}
\end{definition}

 Examples of entropy type functionals $\mathcal{U}_\nu$ satisfying parts 1 of Condition~\ref{as: properties of U} and the $k$-displacement convexity are found on metric measure spaces satisfying the so-called $CD(K,N)$ condition \cite{sturm06, sturm06a, lottvillani,V2}; part 2 of Condition~\ref{as: properties of U} in those examples is satisfied if the domain $M$ is a smooth manifold (as we assume in this paper).

\begin{definition}[{\bf $CD(K,N)$ condition; see. e.g \cite[Ch.  14]{V2}}]\label{cdkn}
A (complete) $n$-dimensional Riemannian manifold $M$ equipped with a reference measure $\nu = e^{-V} \vol$, where $V \in C^2 (M)$, is said to satisfy a $CD(K,N)$ condition for $N \in (n, \infty]$, $K \in \R$ if 
$Ric_{N, \nu} \ge K$. Here, 
\begin{align*}
 Ric_{N,\nu} := Ric+\nabla^2 V - \frac{1}{N-n}\nabla V \otimes \nabla V,
 \end{align*}
where,  $\nabla V \otimes \nabla V: T^2M \to \R $ is  defined as 
\begin{align*}
 (\nabla V \otimes \nabla V)_x  (v, w)=(\nabla V(x) \cdot v)(\nabla V(x) \cdot w).
\end{align*}
 
\end{definition}
We note that when  $N=n$, the $CD(K,n)$ condition is simply that the Ricci curvature $Ric \ge K$ is bounded below by $K$.
\begin{example}[{\bf $k$-displacement convex functionals; see, e.g.,  \cite[Ch. 17]{V2}}]\label{ex: k-convex}
A representative functional for displacement convexity is the following: 
\begin{align*}
U_N(r) =  \begin{cases}
   -N(r^{1-1/N}-r)   & \text{$(1< N<\infty)$ }, \\
     r\log r & \text{$(N=\infty)$},
\end{cases}
\end{align*}
with the convention that $U_\infty (0)=0$. 
For $M$ satisfying the $CD(K,N)$ condition, $\mathcal{U}_\nu$ defined as $\int_M U_N (f_\mu) d\nu$, is $k$-displacement convex  
where  the constant $k$ depends on $N, K$ and $\sup_{t\in [0,1]} \|\frac{d\mu_t}{d\nu}\|_\infty$.  In this case,  $k$ and $K$ have the same sign. (See \cite[Theorem 17.15, (17.11) and Exercise 17.23]{V2}). Moreover, $\mathcal{U}_\nu$ satisfies  Condition~\ref{as: properties of U}.
\end{example}

\subsection{Wasserstein Jensen's inequalities for $k$-convex functionals}\label{ss: Wass Jensen}
 We now present one of the main results of this section, which follows easily from the first order balance for functionals  established in  Lemmas~\ref{lem: lim inf U} and \ref{lem: lim inf V W}: 
\begin{theorem}[\bf Wasserstein Jensen's inequality for $k$-displacement convex functionals]\label{thm: k-Jensen} 
 Let $\mathcal{U}_{\nu}: P_{ac}(M) \to \R$ be the entropy type functional given in Definition~\ref{def: functional U}, 
satisfying Condition~\ref{as: properties of U}. Let $\mathcal{V}, \mathcal{W}: P(M) \to \R\cup \{+\infty\}$ be the potential energy and the interaction energy functional, respectively,  given in Definition~\ref{def: potential interaction}, with Lipschitz functions  $\tilde V: M \to \R$ and $\tilde W: M \times M \to \R$.

Let  $\Omega$ be  a probability measure on $P(M)$ and assume that  $\Omega(\mathcal{A}_L ) >0$ for some $L < \infty$.
Let $\bar \mu$ be the Wasserstein barycenter of $\Omega$.

For  $\mathcal{F} = \mathcal{U}_\nu$, $\mathcal{V}$, or $\mathcal{W}$, suppose $\mathcal{F}$ is $k$-displacement convex as in Definition~\ref{def: k-dis convex}. 
Then, we have 
\begin{align}\label{eq: k-Jensen}
 \mathcal{F}(\bar \mu) \le \int_{P(M)} \mathcal{F}[\mu] d\Omega(\mu) - \frac{k}{2} \int_{P(M)}W_2^2(\bar \mu, \mu) d\Omega(\mu).
\end{align}
\end{theorem}
\begin{proof} 
  Note that $\bar \mu$ is absolutely continuous from $\Omega(\mathcal{A}_L) >0$ and Theorem~\ref{thm: ac for general}. Use the notation in Defintiion~\ref{def: dis inter}, letting $\mu_0=\bar \mu$.  Then,
from $k$-displacement convexity, which is translated to $k$-convexity of $t \mapsto \mathcal{F} (\mu_t)$,  we get 
\begin{align*}
\mathcal{F} (\mu) \ge \mathcal{F}_\nu(\bar \mu) + \liminf_{t \to 0+}  \frac{\mathcal{F}[\mu_t] - \mathcal{F}[\bar \mu]}{t}  + \frac{k}{2}W_2^2 (\bar \mu, \mu).
\end{align*}
To finish the proof,  integrate this against $\Omega$ and use Lemma~\ref{lem: lim inf U} for $\mathcal{F}=\mathcal{U}_\nu$, Lemma~\ref{lem: lim inf V W} for $\mathcal{F}= \mathcal{V}, \mathcal{W}$.
\end{proof}

This theorem can be interpreted as a Wasserstein Jensen's inequality for a variety of functionals on Riemannian manifolds, including those listed in Examples~\ref{ex: k-convex}. In particular, we note the following immediate consequence:
\begin{corollary}[\bf Wasserstein Jensen's inequality on $Ric\ge 0$]\label{cor: Jensen's K ge 0 case}
Assume $Ric \geq 0$.  Let $$\mathcal{U}[\mu] :=\int_{M}U(f_\mu(x))dvol(x),$$ where $r \mapsto r^nU(r^{-n})$ is convex nonincreasing.  Then, letting $\bar \mu$ be the barycenter of $\Omega$, and assuming $\Omega(\mathcal{A}_L )>0$, 
we have
\begin{align}
 \mathcal{U}(\bar \mu) \le \int_{P(M)} \mathcal{U}[\mu] d\Omega(\mu).
\end{align}
\end{corollary}

\begin{remark}
The preceding Corollary also follows from a different geometric version of Jensen's  inequality, (see Theorem ~\ref{thm: distortedconvexity} below), which is similar to the distorted displacement convexity in \cite{c-ems} and \cite[Theorem 17.37]{V2}.  The resulting conveixty statement looks less clean, as distortion coefficients show up inside the integrals,  but eliminates the additional $k$ term.

\end{remark}

}

\subsection{Distorted Wasserstein Jensen's inequality}
 In this subsection, we offer an alternate geometric version of Jensen's  inequality; in contract to Theorem~\ref{thm: k-Jensen}, this version eliminates the $k$-term, but introduces generalized distortion coefficients inside the integral.

This is similar to the distorted convexity of \cite{c-ems} and those in \cite[Theorem 17.37]{V2}.
For simplicity, we deal only with the simplest class of functionals.

\begin{theorem}[\textbf{Distorted Wasserstein Jensen's  inequality}]\label{thm: distortedconvexity} Let $$\mathcal{U}[\mu] :=\int_{M}U(f_{\mu}(x))d\vol(x),$$ where $r \mapsto r^nU(r^{-n})$ is convex nonincreasing and $f_\mu(x)$ is the density of the measure $\mu$, with respect to volume. Assume  that $\Omega$  almost every $\mu$ is absolutely continuous with respect to volume and that $\Omega(\mathcal{A}_L) >0$ for some $L$.  Then $\mathcal{U}$ is displacement convex over barycenters with distortion coefficients $\alpha$; that is,

$$
\mathcal{U}(\bar \mu) \leq \int_{P(M)} \int_{M}U \left(\frac{f_{\mu}(x)}{\alpha_{\lambda_{T_\mu^{-1}(x)}}(x)}\right)\alpha_{\lambda_{T_\mu^{-1}(x)}}(x)d\vol(x)d\Omega(\mu)
$$
where $\bar \mu$ is the barycenter of $\Omega$ and $T_\mu$ is the optimal map from $\bar \mu$ to $\mu$.
\end{theorem}

\begin{remark}
Notice that, like Theorem ~\ref{thm: k-Jensen}, Theorem~\ref{thm: distortedconvexity}  implies Corollary~\ref{cor: Jensen's K ge 0 case}. 
 It is not clear to us whether the two upper bounds on $\mathcal{U}[\bar \mu]$ in Theorems~\ref{thm: k-Jensen} and \ref{thm: distortedconvexity} are comparable; even for doublely supported measures $\Omega =(1-t)\delta_{\mu_0}+t\delta_{\mu_1}$, this remains open (see \cite{V2},Open Problem 17.39).
\end{remark}
\begin{proof}[Proof of Theorem~\ref{thm: distortedconvexity}]

 In the following proof, the determinant estimates in Theorem~\ref{thm: detcontrol} play a key role. Note that this result applies as, from the assumption $\Omega(\mathcal{A}_L)>0$ and Theorem~\ref{thm: ac for general}, the Wasserstein barycenter $\bar \mu$ is absolutely continuous.

For each $\mu$, let $T_\mu(x)=\exp_x(\nabla u_\mu(x))$  be the optimal map pushing $\bar \mu$ forward to $\mu$.   For almost every  point $x$, the map  $T_\mu$ satisfies the following Jacobian equation: $$
|\det DT_\mu(x)| = \frac{f_{\bar \mu}(x)}{f_\mu(T_\mu(x))}.
$$
By a change of variables, this implies
\begin{align*}\label{changevar}
&\int_{M}U \left(\frac{f_{\mu}(x)}{\alpha_{\lambda_{T_\mu^{-1}(x)}}(x)}\right) \alpha_{\lambda_{T_\mu^{-1}(x)}}(x)d\vol(x)\\
&=\int_{M}U \left( \frac{f_{\bar \mu}(x)}{\alpha_{\lambda_x}(T_\mu (x))) |\det DT_\mu(x)|}\right){\alpha_{\lambda_{x}}(T_\mu (x))}|\det DT_\mu(x)|d\vol(x).
\end{align*}

Integrating this equation against $\Omega$, using the Fubini-Tonelli theorem as in the proof of Theorem~\ref{thm: k-Jensen}, (the ordinary) Jensen's  inequality and the convexity of $r \mapsto r^nU(r^{-n})$, we obtain
\begin{eqnarray*}
\int_{P(M)}\int_{M}U\left(\frac{f_{\mu}(x)}{\alpha_{\lambda_{T_\mu^{-1}(x)}}(x)}\right)\alpha_{\lambda_{T_\mu^{-1}(x)}}(x)d\vol(x)d\Omega(\mu)\\
\geq \int_{M}U\left( \frac{f_{\bar \mu}(x)}{r^n(x)}\right)r^n(x) d\vol(x),
\end{eqnarray*}
where 
$$
r(x)=\int_{P(M)}\alpha_{\lambda_x}(T_\mu (x))^{1/n} |\det DT_\mu(x)|^{1/n}d\Omega(\mu).
$$  
The monotonicty  of $r \mapsto r^nU(r^{-n})$ and Theorem~\ref{thm: detcontrol} then yield the desired result.
\end{proof}

\section{Curved random Brunn-Minkowski inequality}\label{S:Brunn-Minkowski}
We complete this paper with an application:  the proof of a random Brunn-Minkowski inequality on smooth metric measures spaces, Theorem~\ref{thm: random BM}. Recall the classical Brunn-Minkowski inequality states that for any two  bounded measurable sets $X, Y \subset \R^n$, 
\begin{align*}
|X+Y|^{1/n} \ge |X|^{1/n} + |Y|^{1/n} 
\end{align*}
where
$X+Y =\{ x+y \ | \ x \in X, y \in Y\}$ and $|X|$ denotes the Euclidean volume. 

Optimal transport techniques have been used to extend this inequality to Riemannian manifolds, as well as more general (not necessarily smooth) metric measure spaces, satisfying an appropriate condition on the curvature, known as  the curvature dimension condition $CD(K, N)$. This direction of research, on smooth manifolds, was initiated in \cite{c-ems} and an extensive discussion can be found in \cite{V2}; see, for instance, Theorem 30.7 therein.

On the other hand, one can easily use induction to extend the Euclidean version to the Minkowski sum of several sets $X_1,X_2,...,X_m$, obtaining 
$$
\left|\sum_{i=1}^mX_i\right|^{1/n} \geq \sum_{i=1}^m|X_i|^{1/n}.
$$  
One can  go one step further, and obtain a version for infinitely many (or random) sets on  Euclidean space.  Unlike the extension to finitely many sets, the extension to random sets is nontrivial and is known as Vitale's random Brunn-Minkowski inequality \cite{vitale}; see \cite{vitale} for a precise statement.    

 It is not obvious that these extensions hold on curved spaces; in particular, the barycenter operation is not associative on a Riemannian manifold, and so the simple induction argument implying the multi-set version on Euclidean space  does not carry over to curved spaces. In this section, we use our  Wasserstein Jensen's inequality from Theoerm \ref{thm: k-Jensen} to extend the Brunn-Minkowski inequality to 
interpolations between both \emph{several} and \emph{infinitely many} 
or {\em random} sets on smooth Riemannian manifolds. 
Our main result in this direction requires a bit of terminology; the definitions below are direct extensions from Euclidean space to Riemannian manifolds of the nomenclature of Vitale \cite{vitale}, where the barycenter operation replaces the Euclidean average.

 Let $X$ be a random measurable set on $M$; that is, a  mapping $X: (\mathcal{P}, \Omega) \rightarrow 2^M$ from a probability space $\mathcal{P}$  (equipped with the probability measure $\Omega$) to the set of  measurable  subsets of $M$.  We will assume that $vol(X) >0$ almost surely; we will then consider the random measure associated to $X$, namely $\mu_X = \frac{1_X \nu}{\nu(X)}$, where $\nu$ is a fixed reference measure.   We will assume that the  $\omega \mapsto \mu_{X(\omega)}$ is measurable with respect to weak-* convergence on $P(M)$ (note that this mapping is well defined if we assume $\nu(X) >0$ almost surely, as in Theorem \ref{thm: random BM} below).

We say a measurable mapping $S: \mathcal{P} \rightarrow M$ is a \emph{selection} of $X$ if  $S(\omega) \in X(\omega)$, for $\Omega$-a.e $\omega$, i.e. $S \in X$, almost surely.  For each selection $S$, let $BC(S) \subseteq M$ be the set of barycenters of the measure $S_\#\Omega$ on $M$; that is, the set of minimizers of $z \mapsto \int_{\mathcal{P}} d^2(z,S(\omega))d\Omega(\omega)$. We define an analogue of the Minkowski sum for a random set by
$$
 Z(X) =\{z: z\in BC(S) \text{ for some selection $S$ of $X$}\}. 
$$
 Note that $Z(X) \subset M$ is simply a subset of $M$; that is, it is {\em not} a random set.

If $M$ statisfies a $CD(K,N)$ condition (recall Definition \ref{cdkn}), we also define, for a given random set $X$, 
\begin{align*}
 \alpha(X) = \begin{cases}
   \inf_S  \inf_{z\in M}  \int_{\mathcal{P}} d^2(z,S)d\Omega(S) & \text{ if $K \ge 0$}, \\
   \sup_S  \inf_{z\in M}  \int_{\mathcal{P}} d^2(z,S)d\Omega(S)       & \text{otherwise},
\end{cases}
\end{align*}
where $\inf_S $ and $\sup_S $ denote infimum and supremum over all selections $S$ of $X$, respectively.

\begin{theorem}[\bf Curved Random Brunn-Minkowski on smooth metric measure spaces]\label{thm: random BM}
Let $M$ be a compact, 
smooth Riemannian manifold equipped with a reference probability measure $\nu$ with $d\nu(x) = e^{-V(x)} dvol(x)$, $V \in C^2(M)$.  Assume that $(M, \nu)$ satisfies a  $CD(K,N)$ curvature-dimension condition  (see Definition~\ref{cdkn}). 
 Let $X$ be a random measurable set, and assume that almost surely  $\nu(X) >0$. 
Then,
\begin{enumerate}
\item 
If $N=\infty$, 
\begin{align*}
 \log [\nu(Z(X))] \geq \mathbb{E}(\log [\nu(X)])  + \frac{k}{2}\alpha(X)
\end{align*}
where $k$ is the same constant as in Example~\ref{ex: k-convex};

 \item If $N< \infty$ and $K \ge 0$,
  \begin{equation}
[\nu(Z(X)) ]^{1/N}\geq \mathbb{E}[\nu(X)^{1/N}]
\end{equation}
where $\mathbb{E}$ denotes the expectation with respect to the probability measure $\Omega$.
\end{enumerate}
 Here  $\nu (Z(X))$ should be interpreted as the outer measure of $Z(X)$, if this set is not Borel.
\end{theorem}
\begin{remark}
 This result easily extends to a complete, non-compact  manifold $M$, provided that the random set $X \subset B$ is contained in a fixed, bounded set $B$ almost surely, but our proof does not cover the case where $X$ satisfies the weaker hypothesis of almost sure compactness.
\end{remark}
\begin{proof}
 We present only the  proof of  the assertion 1; the proof of the second assertion is very similar and is omitted.

The proof is a direct application of the $k$-displacement convexity of the functional $\mathcal{U}_\nu$  from Example~\ref{ex: k-convex}
with  $U_\infty(r) = r\log r$ (for assertion 2, we use instead, $U_N(\rho) = -N(r^{1-1/n}-r)$), together with Theorem \ref{thm: k-Jensen}.
We associate with the random set $X$  the (random) measure $\mu_X = \frac{1_X \nu}{\nu(X)}$.  Note that $\mu_X$ is well defined  and  absolutely continuous almost surely, by the assumption $\nu(X)>0$ a.s.  We will apply our  Wasserstein Jensen's  inequality (Theorem~\ref{thm: k-Jensen})  to the push forward of $\Omega$ by the map $\omega \mapsto \mu_{X(\omega)}$, which is a  probability measure on $P(M)$.  
Let  $\bar\mu$ be the Wasserstein barycenter of this measure, which exists uniquely by Theorem~\ref{thm: existence and uniqueness}.   By  Theorem~\ref{thm: ac for general}, $\bar \mu$  is absolutely continuous with respect to volume, and hence with respect to $\nu$ as well; we denote its density with respect to $\nu$ by $f(x)$: $d\bar \mu(x) = f(x)d \nu(x)$. Now, from the Wasserstein Jensen's inequality (Theorem~\ref{thm: k-Jensen}) 
we have
\begin{align}\label{entjen}
 \int_{\spt{\bar \mu}} f (x) \log f(x) d\nu(x) & \le \int_{\mathcal{P}} \int_{X(\omega)} \frac{1}{\nu (X(\omega))} \log \left(\frac{1}{\nu(X(\omega))}\right) d\nu(x) d\Omega(\omega) \\
& \quad - \frac{k}{2} \int_{\mathcal{P}}W_2^2(\bar \mu, \mu_{X(\omega)}) d\Omega(\omega)\nonumber
\end{align}

Applying the  (ordinary) Jensen's inequality for the convex function $r\to r\log r$ to the measure $\frac{1_{\spt \bar \mu} \nu}{\nu (\spt \bar \mu)}$ and noting that $\int_M f(x)d\nu(x) =1$, the left-hand side is bounded from below by 
$\log(\frac{1}{\nu (\spt \bar \mu)}) $:
\begin{align}\label{ordjen}
\log\left(\frac{1}{\nu(\spt{\bar \mu})}\right) \le   \int_{\spt{\bar \mu}} f (x) \log f(x) d\nu(x).
\end{align}
Next, we claim that $\spt(\bar \mu) \subseteq  Z(X)$; this will imply that  $\nu(\spt(\bar \mu))\le \nu(Z(X))$ and therefore 
\begin{equation}\label{bcvsE(X)}
\log\left(\frac{1}{\nu (Z(X))}\right) \le \log\left(\frac{1}{\nu (\spt \bar \mu)}\right) .
\end{equation}

Now fix any $z \in \spt(\bar \mu)$.  Let $T_{\mu_{X(\omega)}}$ be the optimal map from  $\bar \mu$ to  $\mu_{X(\omega)}$, for any $\omega \in \mathcal{P}$ and define the selection $S:\mathcal{P} \rightarrow M$ by $S(\omega) := T_{\mu_{X(\omega)}}(z)$.  Now, by Lemma \ref{riemannianbc}, 
$z$ is a barycenter of the measure $\lambda_z:=(\omega \mapsto T_{\mu_X(\omega)}(z))_\#\Omega = S_\#\Omega$, $\bar \mu$ almost surely; by definition, this means that $z \in BC(S)$, which in turn shows $z \in Z(X)$, $\bar \mu$ almost surely.  Thus,  $\spt \bar \mu \subset Z(X)$, implying \eqref{bcvsE(X)}.

 Moreover, 
\begin{align*}
   \inf_S  \inf_{z\in M}  \int_{\mathcal{P}} d^2(z,S)d\Omega(S) \le \int_{\mathcal{P}}W_2^2(\bar \mu, \mu_X) d\Omega(X)
   \le \sup_S  \inf_{z\in M}  \int_{\mathcal{P}} d^2(z,S)d\Omega(S).
\end{align*}
 The result now follows easily, by combining the preceding inequality with \eqref{entjen}, \eqref{ordjen} and \eqref{bcvsE(X)}.
\end{proof}

As an immediate consequence, we get  the following multi-set Brunn-Minkowski inequality on $M$.\begin{corollary}[\bf Multi-set Brunn-Minkowski on spaces of nonnegative Ricci curvature]\label{cor: special BM} 
 Let $M$ be a Riemannian manifold with a reference measure $\nu$, satisfying the $CD(K,N)$ condition with $K \ge 0$.  Let $A_i \subseteq M$,  $i=1,2,....m$, be bounded sets and set $Z =\{bc_{\lambda}(x_1,...,x_m): x_i \in A_i\}$ be the set of barycenters of points in the $A_i$.  Then,
 \begin{itemize}
 \item 
  if $N<\infty$, 
\begin{align*}
\nu(Z)^{1/N} \geq \sum_{i=1}^m \lambda_i \nu(A_i)^{1/N};
\end{align*}
\item 
 if $N=\infty$, 
\begin{align*}
 \log \nu(Z) \ge \sum_{i=1}^m \lambda_i \log \nu(A_i).
\end{align*}
\end{itemize}
\end{corollary}
Note that on the Euclidean space (with $\nu=\vol$), this corollary follows easily by iterating the classsical Brunn-Minkowski inequality $m$ times, as the barycenter operation  $bc_{\lambda}(x_1,...,x_m) = \sum_i \lambda_i x_i$ is  associative; however, this is not the case on curved spaces.

\bibliographystyle{plain}
\bibliography{biblio}

\end{document}